\documentclass[aip,cha,author-year,graphicx,amsmath,amssymb,reprint]{revtex4-1}

\draft
\usepackage[utf8]{inputenc}
\usepackage[T1]{fontenc}

\usepackage{amsthm}
\usepackage{thmtools}
\usepackage{subcaption}
\usepackage[hidelinks]{hyperref}
\usepackage[nameinlink]{cleveref}

\usepackage{tikz}
\usetikzlibrary{math}
\usetikzlibrary{calc}
\usetikzlibrary{intersections}
\usetikzlibrary{arrows.meta}
\usetikzlibrary{patterns}
\usetikzlibrary{shapes.geometric}
\usetikzlibrary{angles}

\newcommand{\reals}{\ensuremath{\mathbb{R}}}
\newcommand{\norm}[1]{\ensuremath{\left\vert\left\vert#1\right\vert\right\vert}}
\newcommand{\pdiff}[1]{\ensuremath{\frac{\partial}{\partial #1}}}

\newcommand{\vbar}{\ensuremath{\biggr\vert}}
\newcommand{\flow}{\ensuremath{\mathcal{F}}}

\newcommand{\T}[1]{\ensuremath{T_\partial #1}}
\newcommand{\uT}[1]{\ensuremath{T_1 #1}}

\newcommand{\sT}{\ensuremath{\mathcal{T}}}
\newcommand{\wsT}{\ensuremath{\widetilde{\mathcal{T}}}}

\newcommand{\comment}[1]{}

\newtheorem{theorem}{Theorem}
\newtheorem{lemma}[theorem]{Lemma}
\newtheorem{proposition}[theorem]{Proposition}
\newtheorem{corollary}{Corollary}[theorem]

\theoremstyle{definition}
\newtheorem{example}{Example}
\newtheorem*{remark}{Remark}

\begin{document}
\date{5 December 2022}
\title{Recovering Obstacles from their Travelling Times}
\author{T. Gurfinkel}
\email{tal.gurfinkel@research.uwa.edu.au}
\author{L. Noakes}
\email{lyle.noakes@uwa.edu.au}
\author{L. Stoyanov}
\email{luchezar.stoyanov@uwa.edu.au}
\affiliation{ 
University of Western Australia, Crawley, 6009 WA, Australia
}
	\begin{abstract}
		\cite{math9192434} introduced a method of recovering strictly convex planar obstacles from their set of travelling times. We provide an extension of this construction for obstacles on Riemannian surfaces under some general curvature conditions. It is required that no smooth geodesic intersect more than two obstacles.
	\end{abstract}
	\maketitle
	\begin{quotation}
	Consider a strictly convex obstacle $K$ on a Riemannian surface $M$. $K$ may be a disjoint union of finitely many strictly convex submanifolds of $M$ of dimension 2 with smooth boundary. The set of travelling times of $K$ are determined via scattering of geodesics in $M$, reflecting on $\partial K$, which emanate from an arbitrary strictly convex smooth curve $\partial S$, which bounds $K$ in $M$. These geodesics may approximate light, pressure or other kinds of waves in a uniform medium, which reflect elastically on the body $K$. The lengths of those geodesics beginning and ending in $\partial S$ are called their travelling times, and they form the data which we are given about $K$. In this paper we give a constructive method to recover $K$ from its set of travelling times. We do so by constructing envelopes of smooth geodesics tangent to $K$, which can be found by closely inspecting the set of travelling times. The components of $K$ have to be arranged such that no smooth geodesic intersects more than two of the components. Otherwise ensuring the uniqueness of the recovered obstacles becomes prohibitively difficult, as determining which component of $K$ the constructed envelope belongs to is not possible in general.
 	\end{quotation}
 	\section{Introduction}
 	For the past 50 years, chaotic mathematical billiards have been studied extensively as an intriguing formulation of physical processes involving hard-ball elastic collisions. These are idealised systems involving a point particle moving at constant speed within a container, which reflects upon contact with the boundary of the container, according to the natural law ``angle of reflection is equal to angle of incidence''. The most famous example is Sinai's dispersing billiards on the torus, defined in his seminal paper (cf. \cite{MR0274721}). Since then there have been numerous to the shape of the containers of these billiards, which alter the dynamics wildly while still exhibiting similar chaotic properties, as outlined by \cite{CHERNOVCHAOTICBILLIARDS}. Generally these involve strictly convex boundaries which are the crucial element in inducing the chaotic behaviour of the dynamics. More recently variations on the equations of motion of the particle have also been considered, via introducing curvature to the space in which the particle moves, or in some cases introducing a magnetic field orthogonal to the space, acting on the particle (\cite{MR1382763,PhysRevE.67.065202}). In this paper we consider an inverse problem in mathematical billiards. Suppose, for example, we are tasked with extracting an ore body from a uniform medium, while minimising the amount of excess matter taken along with the valuable ore body. We may set off a series of small charges in an boundary around the ore body, recording the time the pressure waves from each explosion take to return to the boundary, as well as the location. Using this data we aim to recover the shape of the ore body, including the number of components it is made of, as well as their sizes and distances separating them from one another. We formulate this problem mathematically, approximating the pressure waves as point particles, and show that under certain conditions one can recover the ore body (or more generally obstacle) exactly.
 	 	
	A more precise formulation of the problem is as follows. Let $M$ be a geodesically complete, 2-dimensional Riemannian manifold with injectivity radius $\rho > 0$. We say that a 2-dimensional submanifold $W$ of $M$ is strictly convex in $M$ if the following two conditions are satisfied:
	\begin{itemize}
		\item Given any two points $p,q\in W$, the smooth geodesic $\gamma$ from $p$ to $q$ is contained entirely in $W$.
		\item The curvature of the boundary $\partial W$ is positive.
	\end{itemize}
	Let $S$ be a 2-dimensional, strictly convex, compact submanifold of $M$ with smooth boundary and diameter smaller than $\rho$. This implies that between any two points in $S$ there is a unique smooth geodesic in $S$.
	Suppose $K = K_1 \cup \dots\cup K_n$ is a union of $n\geq 2$\footnote{For the case where $n=2$ there is a simpler procedure for reconstructing the obstacles given by \cite{MR3351977}.} disjoint 2-dimensional, strictly convex submanifolds $K_i$ of $M$ with smooth boundary, contained within $S$. Denote $S_K = \overline{S\backslash K}$. 
	Let $\kappa_S$ and $\kappa_K$ be the maximal and minimal (respectively) sectional curvatures of $S$ and $\partial K$ respectively. Suppose that either $\kappa_S<0$ or,
	\[ \kappa_S>0,\quad \rho\sqrt{\kappa_S}<\frac{\pi}{2},\text{ and }\sqrt{\kappa_S}\tan(\rho\kappa_S)<\kappa_K.\]
	These conditions ensure that convex fronts will remain convex after propagation via the billiard flow, see \cite{MR807598,MR1022522}.
	We consider (generalised) geodesics to be piecewise-smooth, constant speed curves in $S_K$ which reflect on $K$ according to the usual reflection law, and such that any smooth arc of the curve is a Riemannian geodesic.  Namely, a Riemannian geodesic is a critical point of the energy functional $E$, where $\gamma$ does not intersect $\text{Int } K$. \[E(\gamma) = \int_a^b \frac{1}{2}\norm{\dot\gamma(t)}^2\ dt.\]
	This defines the billiard dynamical system on our curved billiard table $S$. We denote the set of travelling times for $K$ by $\sT$. That is, the set of all triples $(x,y,t)$ such that $x,y\in\partial S$ and $t$ is the length of some geodesic $\gamma$ between $x$ and $y$.
	
	\begin{figure*}
	\center
		\subcaptionbox{3 obstacles in general position.\label{fig:yes_general_position}}{
		\begin{tikzpicture}[rotate=0]
			\tikzmath{
			    \r1 = 2cm;%radius of smaller circle
			    \d1 = 4cm; %diameter of smaller circle
			    coordinate \c;
			    \c1 = (2,0);
			    \c2 = (-8,4);
			    \c3 = (-2,7);
			}
			
			\node[circle,draw] (c1) at (\c1) [minimum size = \d1]{$K_l$};
			\coordinate (c1a) at ($(\c1)+(2,0)$);
			\coordinate (c1b) at ($(\c1)+(-2,0)$);
			
			\node[circle,draw] (c2) at (\c2) [minimum size = \d1]{$K_i$};
			\coordinate (c2a) at ($(\c2)+(-2,0)$);
			\coordinate (c2b) at ($(\c2)+(2,0)$);
			
			\node[circle,draw] (c3) at (\c3) [minimum size = \d1]{$K_j$};
			\coordinate (c2a) at ($(\c2)+(-2,0)$);
			\coordinate (c2b) at ($(\c2)+(2,0)$);
			
			\coordinate (a) at (4,2);
			\coordinate (b) at (-10,2);
			\coordinate (ta1) at (tangent cs:node=c1,point={(a)},solution=1);
			\coordinate (ta2) at (tangent cs:node=c1,point={(a)},solution=2);
			\coordinate (tb1) at (tangent cs:node=c1,point={(b)},solution=1);
			\coordinate (tb2) at (tangent cs:node=c1,point={(b)},solution=2);
			\coordinate (sa1) at (tangent cs:node=c2,point={(a)},solution=1);
			\coordinate (sa2) at (tangent cs:node=c2,point={(a)},solution=2);
			\coordinate (sb1) at (tangent cs:node=c2,point={(b)},solution=1);
			\coordinate (sb2) at (tangent cs:node=c2,point={(b)},solution=2);
			
			\coordinate (i1) at ($(c2) + (68:2)$);
			\coordinate (til1) at (tangent cs:node=c1,point={(i1)},solution=2);
			\draw[dashed] (i1) -- (til1);
			
			\coordinate (i2) at ($(c2) + (-112:2)$);
			\coordinate (til2) at (tangent cs:node=c1,point={(i2)},solution=1);
			\draw[dashed] (i2) -- (til2);
			
			\coordinate (i3) at ($(c2) + (118:2)$);
			\coordinate (tij1) at (tangent cs:node=c3,point={(i3)},solution=2);
			\draw[dash dot] (i3) -- (tij1);
			
			\coordinate (i4) at ($(c2) + (-65:2)$);
			\coordinate (tij2) at (tangent cs:node=c3,point={(i4)},solution=1);
			\draw[dash dot] (i4) -- (tij2);
			
			\coordinate (j1) at ($(c3) + (29:2)$);
			\coordinate (tjl1) at (tangent cs:node=c1,point={(j1)},solution=2);
			\draw[dotted] (j1) -- (tjl1);
			
			\coordinate (j2) at ($(c3) + (-150:2)$);
			\coordinate (tjl2) at (tangent cs:node=c1,point={(j2)},solution=1);
			\draw[dotted] (j2) -- (tjl2);
			
		\end{tikzpicture}
		}
		
		\subcaptionbox{3 obstacles not in general position.\label{fig:not_general_position}}{
		\begin{tikzpicture}[rotate=0]
			
			\tikzmath{
			    \r1 = 2cm;%radius of smaller circle
			    \d1 = 4cm; %diameter of smaller circle
			    coordinate \c;
			    \c1 = (2,0);
			    \c2 = (-8,4);
			    \c3 = (-2,4);
			}
			
			\node[circle,draw] (c1) at (\c1) [minimum size = \d1]{$K_l$};
			\coordinate (c1a) at ($(\c1)+(2,0)$);
			\coordinate (c1b) at ($(\c1)+(-2,0)$);
			
			\node[circle,draw] (c2) at (\c2) [minimum size = \d1]{$K_i$};
			\coordinate (c2a) at ($(\c2)+(-2,0)$);
			\coordinate (c2b) at ($(\c2)+(2,0)$);
			
			\node[circle,draw] (c3) at (\c3) [minimum size = \d1]{$K_j$};
			\coordinate (c2a) at ($(\c2)+(-2,0)$);
			\coordinate (c2b) at ($(\c2)+(2,0)$);
			
			\coordinate (a) at (4,2);
			\coordinate (b) at (-10,2);
			\coordinate (ta1) at (tangent cs:node=c1,point={(a)},solution=1);
			\coordinate (ta2) at (tangent cs:node=c1,point={(a)},solution=2);
			\coordinate (tb1) at (tangent cs:node=c1,point={(b)},solution=1);
			\coordinate (tb2) at (tangent cs:node=c1,point={(b)},solution=2);
			\coordinate (sa1) at (tangent cs:node=c2,point={(a)},solution=1);
			\coordinate (sa2) at (tangent cs:node=c2,point={(a)},solution=2);
			\coordinate (sb1) at (tangent cs:node=c2,point={(b)},solution=1);
			\coordinate (sb2) at (tangent cs:node=c2,point={(b)},solution=2);
			
			\coordinate (i1) at ($(c2) + (68:2)$);
			\coordinate (til1) at (tangent cs:node=c1,point={(i1)},solution=2);
			\draw[dashed] (i1) -- (til1);
			
			\coordinate (i2) at ($(c2) + (-112:2)$);
			\coordinate (til2) at (tangent cs:node=c1,point={(i2)},solution=1);
			\draw[dashed] (i2) -- (til2);

			%\draw[-{Latex[length=5mm,width=2mm]}] (b) -- (a);
			
		\end{tikzpicture}
		}
		
		\caption{Qualitatively, general position ensures that any two obstacles cannot 'obscure' a third obstacle.}\label{fig:general_position}
	\end{figure*}
	
	In our curved billiard setting we consider the inverse problem of recovering $K$ from its set of travelling times, as given by the billiard flow.
	The travelling times uniquely determine the obstacle $K$ when $M$ is Euclidean space of dimension 3 or larger, as shown by \cite{MR3359579}. Moreover, one is able to recover the volume of $K$ using a generalisation of Santalo's formula (cf. \cite{STOYANOV20172991}) when $M$ is a Riemannian manifold of any dimension at least 2. Recently an approach by \cite{MR4449325} using the so called layered-scattering technique allows one to determine the volume and curvature of obstacles with codimension 2 or greater.
	However, none of these approaches allows us to constructively recover the obstacle $K$ from the travelling times $\sT$. A novel constructive method of recovering obstacles in Euclidean 2-space was given in \cite{math9192434}, provided the the connected components $K_i$ are in \emph{general position} (\Cref{fig:general_position}). That is, no smooth geodesic in $M$ will intersect more than two of the components $K_i$.
	We note that the general position condition is equivalent to requiring that Ikawa's no-eclipse condition is satisfied (cf. \cite{MR949013}). The most clear and immediate consequence of this condition is the following fact:

	\begin{remark}
		If $\gamma$ is a geodesic tangent to $K$ at some point $\gamma(t^*)$, then $\gamma(t^*)$ is either the first or last reflection point of $\gamma$. To show that this holds, consider the case where $\gamma(t^*)$ is not the first or last reflection point of $\gamma$ (\Cref{fig:impossible_gen_pos}). Let $\gamma(t^*_{-1})$ and $\gamma(t^*_{+1})$ be the points of reflection before and after $\gamma(t^*)$. Then the segment $\gamma|_{[t^*_{-1},t^*_{+1}]}$ is a smooth geodesic which intersects three obstacles, a contradiction to the general position condition.
	\end{remark}
	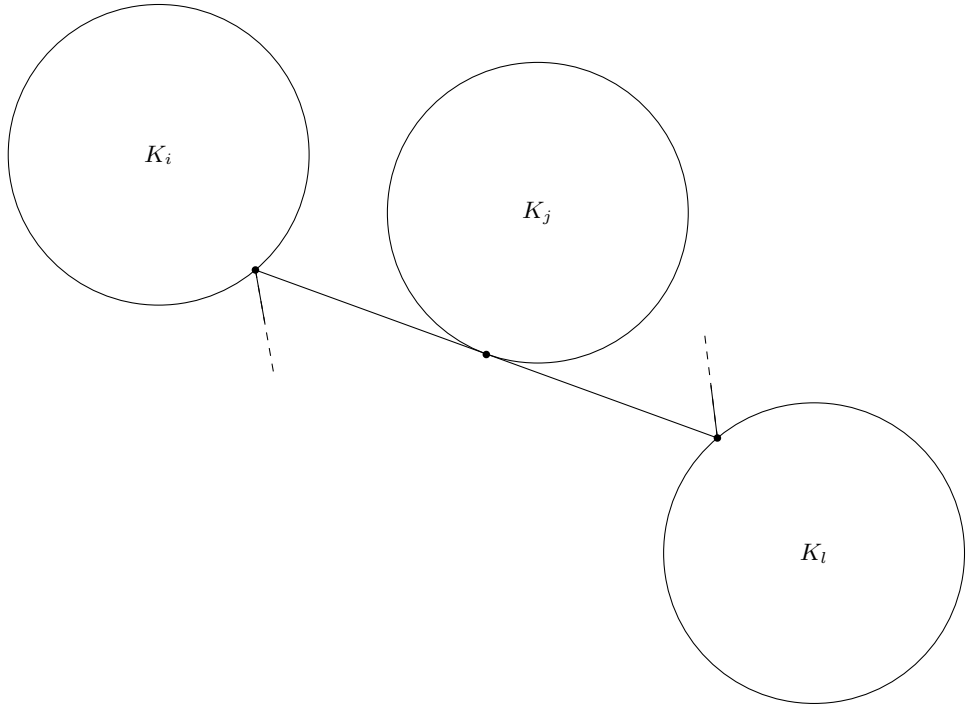
\begin{figure*}
		\begin{tikzpicture}[rotate=-20]
			\tikzmath{
			    \r1 = 2cm;%radius of smaller circle
			    \d1 = 4cm; %diameter of smaller circle
			    coordinate \c;
			    \c1 = (2,0);
			    \c2 = (-8,2);
			    \c3 = (-3,3);
			}
			
			\node[circle,draw, name path = c1] (c1) at (\c1) [minimum size = \d1]{$K_l$};
			\coordinate (c1a) at ($(\c1)+(2,0)$);
			\coordinate (c1b) at ($(\c1)+(-2,0)$);
			
			\node[circle,draw, name path = c2] (c2) at (\c2) [minimum size = \d1]{$K_i$};
			\coordinate (c2a) at ($(\c2)+(-2,0)$);
			\coordinate (c2b) at ($(\c2)+(2,0)$);
			
			\node[circle,draw, name path = c3] (c3) at (\c3) [minimum size = \d1]{$K_j$};
			\coordinate (c2a) at ($(\c2)+(-2,0)$);
			\coordinate (c2b) at ($(\c2)+(2,0)$);
			
			\coordinate (a) at (4,1);
			\coordinate (b) at (-10,1);
			\path[name path=ab] (a) -- (b); 
			\path[name intersections={of=ab and c1,by={sa1,sa2}}];
			\path[name intersections={of=ab and c2,by={sb1,sb2}}];
			\coordinate (ta1) at (tangent cs:node=c3,point={(a)},solution=2);
			\fill (sb2) circle (0.05cm);
			\fill (sa2) circle (0.05cm);
			\fill (ta1) circle (0.05cm);
			
			\draw[dashed, rotate=72] (sa2) -- ($(sa2) + (1,1)$);
			\draw[dashed, rotate=-105] (sb2) -- ($(sb2) + (1,1)$);
			\draw[rotate=72] (sa2) -- ($(sa2) + (0.5,0.5)$);
			\draw[rotate=-105] (sb2) -- ($(sb2) + (0.5,0.5)$);
			\draw (sb2) -- (sa2);
			
		\end{tikzpicture}
		\caption{A geodesic tangent to $K$ with points of reflections both before and after the point of tangency.\label{fig:impossible_gen_pos}}
	\end{figure*}
	In this paper we extend the results of \cite{math9192434} and outline how to reconstruct an obstacle in general position in the more general setting of a curved billiard table. We do so via constructing envelopes of smooth geodesics tangent to the obstacle. Relying on the aforementioned fact in the remark above, we can distinguish which component the geodesics are tangent to. This allows us to piece together the envelopes, ensuring we are correctly reconstructing the obstacle.
	
	Given a point $q=(x,\omega)\in\uT{S_k}$, let $\gamma_q$ be the geodesic uniquely determined by $q$. We say that $\gamma_q$ is \emph{non-trapped} if there exists two distinct times $t_0,t_0'\in\reals$ such that $\gamma_q(t_0),\gamma_q(t_0')\in\partial S$. Otherwise we say that $\gamma_q$ is \emph{trapped}. We denote the set of all points $q\in\uT{S_K}$ such that $\gamma_q$ is trapped by $Trap(S_K)$, called the trapping set of $K$. We cannot hope to recover an obstacle $K$  where the trapping set has nonzero measure, since by definition there will be an open set of points on the boundary of $K$ which cannot be detected in the set of travelling times of $K$.
	It should be noted that the convexity condition on the obstacle $K$ is sufficient to ensure that each connected component has an empty trapping set.
	Obstacles with non-empty trapping sets are known to exist in Euclidean billiard tables of dimension $m\geq 2$. We refer the reader to see \cite{MR3514729} for explicit constructions of such obstacles. One could reasonably conjecture that such trapping obstacles will also exist in a Riemannian manifold. In \Cref{example:trapping-obstacle} we construct a trapping obstacle in a 2-dimensional Riemannian manifold. Since the trapping set of an obstacle is known to be stable under $C^k$ ($k\geq 3$) perturbations \cite{STOYANOV20172991}, one can smoothly perturb one of the obstacles constructed in \cite{MR3514729} to create a trapping obstacle in the Riemannian manifold derived from perturbing the metric in the same manner. There are no known constructions of trapping obstacles for a given Riemannian metric in dimensions higher than 2.
	
	\begin{example}\label{example:trapping-obstacle}
		Suppose $M$ is a Riemannian manifold of dimension 2 with injectivity radius $\rho > 0$. Pick two points $F_1,F_2\in M$ such that $d_g(F_1,F_2)<\rho$. Define the following ellipse-like curve:
	\[
		E = \{x\in M : d_g(x,F_1) + d_g(x,F_2) = r \},
	\]
	where $r\in\reals$ is a constant such that $d_g(F_1,F_2)<r<\rho$. We will call such curves Riemannian ellipses. Note that such curves are in fact strictly convex. Consider the gradient of $d_g(x,F_i),$ $i=1,2:$
	\[
		grad_x\ d_g(x,F_i) = -\dot\gamma_{(x,F_i)}(0),
	\]
	where $\gamma_{(x,F_i)}$ is the geodesic from $x$ to the focus $F_i$. Now we vary $x$ along $E$. Since every point in $E$ satisfies the equation
	\[
		d_g(x_h,F_1) + d_g(x_h,F_2) = r,
	\]
	we may take the derivative to find the following relation:
	\[
		\langle x_h'(0),\dot\gamma_{(x,F_1)}(0) + \dot\gamma_{(x,F_2)}(0)\rangle = 0.
	\]
	Thus, any billiard ray crossing through one of the foci must also cross through the other  (\Cref{fig:ellipce_foci_reflection}). To construct a trapping obstacle from $E$, we adapt Livshits' classical construction using a Euclidean ellipse. Consider the unique smooth geodesic $\gamma$ from $F_1$ to $F_2$. Extending this geodesic from $F_1$, denote the first (possibly only) intersection of $\gamma$ with $E$ by $A_1$. Similarly extending $\gamma$ from $F_2$ we find the first intersection with $E$ by $A_2$. 
	Now consider any ray $\eta$ which crosses $\gamma$ between $F_1$ and $F_2$, let $x\in E$ be the point of reflection of $\eta$ on $E$. Consider the billiard ray $\beta$ from $F_1$ to $x$. By our previous argument, it follows that $\beta$ must cross through $F_2$ after reflecting at $x$. Taking a normal coordinate chart about $x$, large enough to contain $E$, we see that $\beta$ is composed of two straight lines, from $F_1$ to $x$ to $F_2$, and $\eta$ is a straight line from $x$ intersecting $\gamma$ between the foci. Since the angle of reflection of $\eta$ at $x$ is smaller than the angle of reflection of $\beta$ it follows that the reflected ray $\eta'$ of $\eta$ will also fall between the two straight lines of $\beta$. Now since the exponential map is injective within our chart, $\eta'$ cannot intersect $\beta$ and hence must intersect $\gamma$ between the two foci $F_1$ and $F_2$. 
	
	Since it is homeomorphic to $\mathbb{S}^1$, the Riemannian ellipse $E$ can be split into two curves $l_1,l_2$, from $A_1$ to $A_2$. Extend $l_1$ (\Cref{fig:trapping_obstacle_poincare}) from $A_1$ to $A_2$ to create a closed curve homeomorphic to $\mathbb{S}^1$, such that:
	\begin{itemize}
		\item $l_1$ does not intersect $\gamma$ between $A_1$ and $A_2$ except at $F_1$ and $F_2$, where it is tangent to $\gamma$
		\item $l_1$ does not obstruct rays from reaching $\gamma$ between $F_1$ and $F_2$.
	\end{itemize}
	  It follows that $l_1$ has a trapping set of positive measure, since any ray with a reflection point on $l_1$ between $A_1$ and $F_1$ must then reflect between $A_2$ and $F_2$, and vice versa by our argument above.
	\begin{figure}[h]
		\center
		\includegraphics[width=0.9\linewidth,trim={20px 40px 50px 90px},clip]{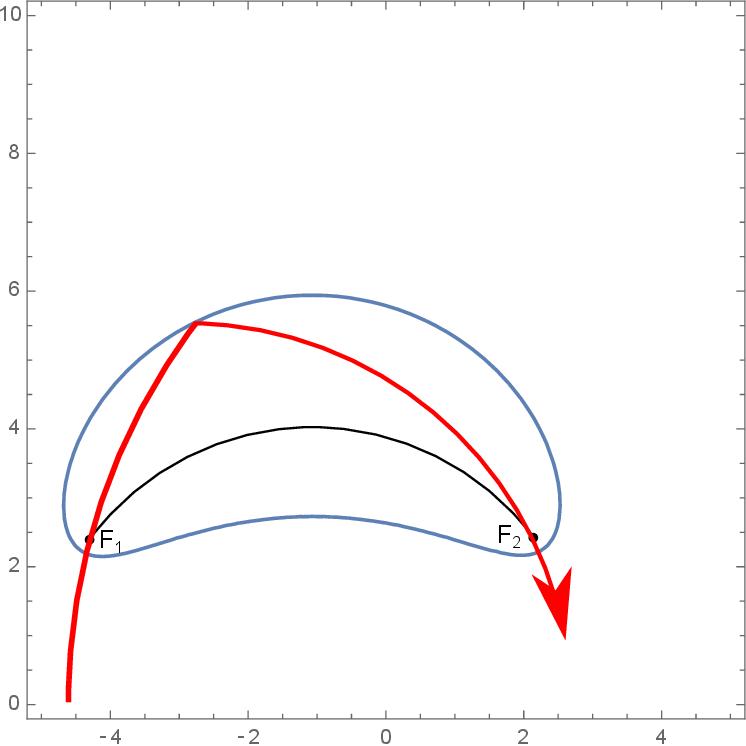}
		\caption{A billiard ray through the foci of a Riemannian $E$ in the Poincare half-plane.}
		\label{fig:ellipce_foci_reflection}
	\end{figure}
	\begin{figure}[h]
		\center
		\includegraphics[width=0.9\linewidth,trim={0px 0px 0px 0px},clip]{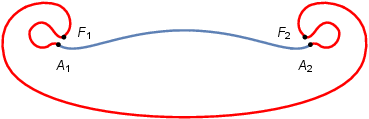}
		\caption{Half of $E$ (\Cref{fig:ellipce_foci_reflection}, in blue) extended (in red) to form a trapping obstacle.}
		\label{fig:trapping_obstacle_poincare}
	\end{figure}
	\qed
	\end{example}
	\section{Recovering the number of components in $K$}
	Define the following submanifolds of the unit tangent bundles, $\uT{S}$ and $\uT{S_K}$, of $S$ and $S_K$ respectively: \begin{equation}
			\T{S} = \{(x,\omega)\in \uT{S}:x\in \partial S\}
		\end{equation}\begin{equation}
			\T{S_K} = \{(x,\omega)\in \uT{S_K}:x\in \partial S_K\}
		\end{equation}
	Let $\flow_t:\uT{S}\times\reals\to\uT{S}$ be the smooth geodesic flow (see e.g. \cite{GEOFLOW} for a definition).
	\begin{lemma}\label{lemma:singlereflection}
	Suppose that $\gamma$ is a smooth geodesic from $x_0\in \partial S_K$ to $y_0\in \partial S_K$, such that $\gamma$ is not tangent to $\partial S_K$ at any point and doesn't intersect $\partial S_K$ at any points except $x_0$ and $y_0$. If $\omega_0$ is the initial direction of $\gamma$ then there are neighbourhoods $W$ of $(x_0,\omega_0)$ in $T_{1}S_K$ and $V$ of $y_0$ in $\partial S_K$, and unique smooth functions $y:W\to V$ and $\tau:W\to\reals^+$ such that $y(x,\omega) = \pi_1\circ\flow_{\tau(x,\omega)}(x,\omega)$ for all $(x,\omega)\in W$.
	\end{lemma}
	\begin{proof}
		 Let $\phi$ be a local defining function for $\partial S_K$ defined on a neighbourhood $U$ of $y_0$. That is, $\phi^{-1}(0) = U\cap \partial S_K$. Let \[Y(x,\omega,t) = \pi_1\circ \flow_t(x,\omega),\] restricted to $\T{S_K}\times \reals$, then $Y$ is well defined since $M$ is geodesically complete, and smooth by definition. Finally, let $t_0>0$ be such that $Y(x_0,\omega_0,t_0)=y_0$ and note that $\phi(y_0) = 0$. Consider the composition $\phi\circ Y$ and its derivative with respect to $t$ at $t_0$. Suppose that we fix $x'$, $\omega'$ and $t'$ such that ${\phi\circ Y(x',\omega',t') = 0}$ and 
	\[\pdiff{t}(\phi\circ Y(x',\omega',t))\vbar_{t=t'}=0.\]
	Which implies that
	\[d\phi_{Y(x',\omega',t')}(d_tY(x',\omega',t')) = 0.\]
	Then $d_tY$ at $(x',\omega',t')$ is in the kernel of $d\phi_{Y(x',\omega',t')}$. But \[\ker d\phi_{Y(x',\omega',t')} = T_{Y(x',\omega',t')}\partial S_K.\] So the geodesic $Y(x',\omega', t)$ is tangent to $\partial S_K$ at $Y(x',\omega',t')$. By assumption, $Y(x_0,\omega_0,t)$ is nowhere tangent to $\partial S_K$. So 
	\[\pdiff{t}(\phi\circ Y(x_0,\omega_0,t))\vbar_{t=t_0}\neq 0.\] 
	Therefore we can apply the implicit function theorem to the function $\phi\circ Y$ at the point $y_0$ to find a unique function $\tau$ defined on a neighbourhood $W$ of $(x_0,\omega_0)$ such that $\phi\circ Y(x,\omega,\tau(x,\omega)) = 0$ for all $(x,\omega)\in W$ and $\tau(x_0,\omega_0) = t_0$.
	
	Now $\flow_{\tau(x,\omega)}$ is a diffeomorphism from $W$ onto its image $\widetilde V = \flow_{\tau(x,\omega)}(W)$. Let $V = \pi_1(\widetilde V)$, then $Y$ maps $W$ onto $V$ by definition. Finally, let $y(x,\omega) = Y(x,\omega,\tau(x,\omega))$, then $y:W\to V$ is the desired map.
	\end{proof}
	
	\begin{remark}
		Our assumption that $\text{diam}(S)<\rho$ allows us to find a normal coordinate neighbourhood $U\supset S$ centred around any point $x\in S$. It follows that any two non-reflecting geodesics in $S$ cannot intersect more than once. We use this fact to prove the results.
	\end{remark}

	\begin{lemma}\label{lemma:twocomponents}
		Suppose that $\alpha:[0,1]\to S$ is a smooth geodesic such that $\alpha(0),\alpha(1)\in\partial S$. Then $S\backslash\alpha([0,1])$ has two connected components.
	\end{lemma}
	\begin{proof}
	Take a normal coordinate chart $\psi: U\to V\subseteq\reals^2$ about $\alpha(0)$ such that $S\subseteq U$. Then $\psi(\alpha([0,1]))$ is a straight line in $U$. Now, $\partial S\backslash\{\alpha(0),\alpha(1)\}$ has two path components, say $\partial S_1$ and $\partial S_2$. Let $\beta = \partial S_1\cup\alpha([0,1])$. Then $\psi(\beta)$ is a Jordan curve in $V$. So by the Jordan curve theorem $V\backslash\psi(\beta)$ has two connected components. i.e. $\psi(S\backslash\alpha([0,1]))$ has exactly two connected components, one of which is bounded by $\psi(\beta)$.
	\end{proof}
	
	\begin{figure*}
		\center
		\begin{tikzpicture}
			
			\tikzmath{
			    \d1 = 10cm;%diameter of larger circle
			    \d2 = 4cm; %diameter of smaller circle
			    coordinate \c;
			    \c1 = (0,0);
			    \c2 = (-0.2,-0.4);
			}
			
			%\draw[help lines] (-6,-6) grid (6,6);
			
			\coordinate (a) at (0,5);
			
			\node[circle,draw, name path = circle1] (c1) at (\c1) [minimum size = \d1] {};
			\fill (a) circle (0.1cm) node[above]{$x_1$};
			\node[circle,draw] (c2) at (\c2) [minimum size = \d2]{};
			
			\coordinate (b1) at (tangent cs:node=c2,point={(a)},solution=1);
			\coordinate (b2) at (tangent cs:node=c2,point={(a)},solution=2);
			
			\coordinate (d1) at ($2*(b1)-(a)$);
			\coordinate (d2) at ($2*(b2)-(a)$);
			
			\draw[opacity=0, name path = t1] (a) -- (b1) -- (d1);
			\draw[-{Latex[length=5mm,width=2mm]}](a)--(b1) node[left, midway] {$\gamma_1$};
			\draw[name path = tprime, name intersections = {of = t1 and circle1, by={da1,da2}}] (a) -- (da2);
			\fill (da2) circle (0.1cm) node[left]{$y_1$};
			
			\draw[opacity=0, name path = t2] (a) -- (b2) -- (d2);
			\draw[blue, dashed,name intersections = {of = t2 and circle1, by={db1,db2}}] (a) -- (db2) node[right, near start] {$\alpha$};
			\fill[blue] (db2) circle (0.1cm) node[blue, right]{$\alpha(1)$};
			\begin{scope}
				\clip (a) -- (-6,6) -- (-6,-6) -- (da2)--cycle;
				\draw [pattern=north west lines, pattern color=gray, opacity = 0.7] (\c1) circle (5cm);
			\end{scope}
			\begin{scope}			
				\clip (a) -- (6,6) -- (6,-6) -- (0,-6) -- (d1)--cycle;
				\draw [pattern=dots, pattern color=gray, opacity = 0.4] (\c1) circle (5cm);
			\end{scope}
			\fill[white] (\c2) circle (2cm);
			\node at (\c2) {$K_l$};
			\node at (-3,2) {$\widetilde S_2$};
			\node at (2,-3) {$\widetilde S_1$};
			
		\end{tikzpicture}
		\caption{$S$ is split into two components $\widetilde S_1$ (dotted area) and $\widetilde S_2$ (dashed aread).}\label{fig:components}
	\end{figure*}
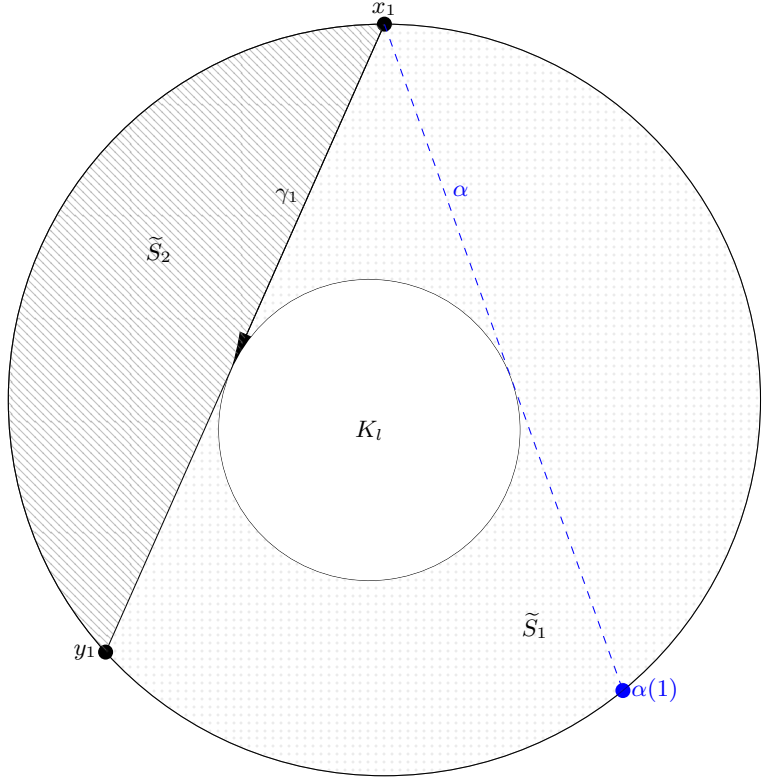
	
	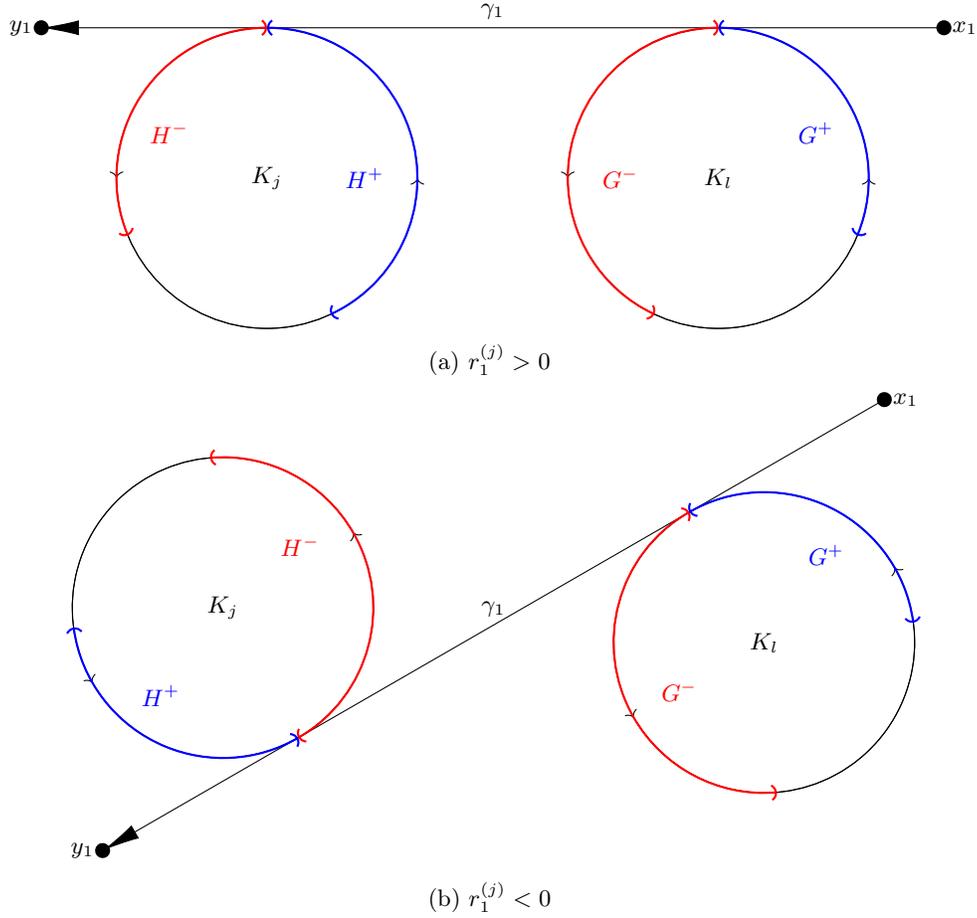
\begin{figure*}
	\center
		\subcaptionbox{$r_1^{(j)}>0$\label{fig:gplusminus_positive}}{
		\begin{tikzpicture}
			\tikzmath{
			    \r1 = 2cm;%radius of smaller circle
			    \d1 = 4cm; %diameter of smaller circle
			    coordinate \c;
			    \c1 = (3,0);
			    \c2 = (-3,0);
			}
			
			\node[circle,draw] (c1) at (\c1) [minimum size = \d1]{$K_l$};
			\coordinate (c1a) at ($(\c1)+(2,0)$);
			\coordinate (c1b) at ($(\c1)+(-2,0)$);
			
			\node[circle,draw] (c2) at (\c2) [minimum size = \d1]{$K_j$};
			\coordinate (c2a) at ($(\c2)+(-2,0)$);
			\coordinate (c2b) at ($(\c2)+(2,0)$);
			
			\coordinate (a) at (6,2);
			\coordinate (b) at (-6,2);
			\coordinate (ta1) at (tangent cs:node=c1,point={(a)},solution=1);
			\coordinate (ta2) at (tangent cs:node=c1,point={(a)},solution=2);
			\coordinate (tb1) at (tangent cs:node=c1,point={(b)},solution=1);
			\coordinate (tb2) at (tangent cs:node=c1,point={(b)},solution=2);
			\coordinate (sa1) at (tangent cs:node=c2,point={(a)},solution=1);
			\coordinate (sa2) at (tangent cs:node=c2,point={(a)},solution=2);
			\coordinate (sb1) at (tangent cs:node=c2,point={(b)},solution=1);
			\coordinate (sb2) at (tangent cs:node=c2,point={(b)},solution=2);
			
			\draw[-{Latex[length=5mm,width=2mm]}] (a) -- (b) node[midway, above]{$\gamma_1$};
			\fill (a) circle (0.1cm) node[right]{$x_1$};
			\fill (b) circle (0.1cm) node[left]{$y_1$};
			
			%\draw[dashed] (a) -- (ta2);
			%\draw[dashed] (b) -- (tb1);
			%\draw[dashed] (a) -- (sa2);
			%\draw[dashed] (b) -- (sb1);
			
			\pic [draw, angle radius=2cm,->] {angle=c1a--c1--c1b};
			\pic [draw, angle radius=2cm,->] {angle=c1b--c1--c1a};
			\pic [draw, red, angle radius=2cm, {Parenthesis}-{Parenthesis},thick] {angle=ta1--c1--tb1};
			\pic [draw, blue, angle radius=2cm, {Parenthesis}-{Parenthesis},thick] {angle=ta2--c1--ta1};
			\node[blue] (g+) at (4.3,0.6) {$G^+$};
			\node[red] (g-) at (1.7,0) {$G^-$};

			\pic [draw, angle radius=2cm,->] {angle=c2a--c2--c2b};
			\pic [draw, angle radius=2cm,->] {angle=c2b--c2--c2a};
			\pic [draw, red, angle radius=2cm, {Parenthesis}-{Parenthesis},thick] {angle=sa1--c2--sb1};
			\pic [draw, blue, angle radius=2cm, {Parenthesis}-{Parenthesis},thick] {angle=sa2--c2--sa1};
			\node[blue] (h+) at (-1.7,0) {$H^+$};
			\node[red] (h-) at (-4.3,0.6) {$H^-$};
			
		\end{tikzpicture}
		}
		
		\subcaptionbox{$r_1^{(j)}<0$\label{fig:gplusminus_negative}}{
		\begin{tikzpicture}[rotate=30]
			
			\tikzmath{
			    \r1 = 2cm;%radius of smaller circle
			    \d1 = 4cm; %diameter of smaller circle
			    coordinate \c;
			    \c1 = (3,0);
			    \c2 = (-3,4);
			}
			
			\node[circle,draw] (c1) at (\c1) [minimum size = \d1]{$K_l$};
			\coordinate (c1a) at ($(\c1)+(2,0)$);
			\coordinate (c1b) at ($(\c1)+(-2,0)$);
			
			\node[circle,draw] (c2) at (\c2) [minimum size = \d1]{$K_j$};
			\coordinate (c2a) at ($(\c2)+(-2,0)$);
			\coordinate (c2b) at ($(\c2)+(2,0)$);
			
			\coordinate (a) at (6,2);
			\coordinate (b) at (-6,2);
			\coordinate (ta1) at (tangent cs:node=c1,point={(a)},solution=1);
			\coordinate (ta2) at (tangent cs:node=c1,point={(a)},solution=2);
			\coordinate (tb1) at (tangent cs:node=c1,point={(b)},solution=1);
			\coordinate (tb2) at (tangent cs:node=c1,point={(b)},solution=2);
			\coordinate (sa1) at (tangent cs:node=c2,point={(a)},solution=1);
			\coordinate (sa2) at (tangent cs:node=c2,point={(a)},solution=2);
			\coordinate (sb1) at (tangent cs:node=c2,point={(b)},solution=1);
			\coordinate (sb2) at (tangent cs:node=c2,point={(b)},solution=2);
			
			\draw[-{Latex[length=5mm,width=2mm]}] (a) -- (b) node[midway, above]{$\gamma_1$};
			\fill (a) circle (0.1cm) node[right]{$x_1$};
			\fill (b) circle (0.1cm) node[left]{$y_1$};
			
			%\draw[dashed] (a) -- (ta2);
			%\draw[dashed] (b) -- (tb1);
			%\draw[dashed] (a) -- (sa1);
			%\draw[dashed] (b) -- (sb2);
			
			\pic [draw, angle radius=2cm,->] {angle=c1a--c1--c1b};
			\pic [draw, angle radius=2cm,->] {angle=c1b--c1--c1a};
			\pic [draw, red, angle radius=2cm, {Parenthesis}-{Parenthesis},thick] {angle=ta1--c1--tb1};
			\pic [draw, blue, angle radius=2cm, {Parenthesis}-{Parenthesis},thick] {angle=ta2--c1--ta1};
			\node[blue] (g+) at (4.3,0.6) {$G^+$};
			\node[red] (g-) at (1.7,0) {$G^-$};

			\pic [draw, angle radius=2cm,->] {angle=c2a--c2--c2b};
			\pic [draw, angle radius=2cm,->] {angle=c2b--c2--c2a};
			\pic [draw, red, angle radius=2cm, {Parenthesis}-{Parenthesis},thick] {angle=sa2--c2--sa1};
			\pic [draw, blue, angle radius=2cm, {Parenthesis}-{Parenthesis},thick] {angle=sb2--c2--sa2};
			\node[blue] (h+) at (-4.3,3.4) {$H^+$};
			\node[red] (h-) at (-1.7,4.2) {$H^-$};
		\end{tikzpicture}
		}
		
		\caption{We can identify the sets of geodesics $G^+(\gamma_1)$ and $G^-(\gamma_1)$ with arcs, $G^+$ (blue) and $G^-$ (red), respectively, of $\partial K_l$.}\label{fig:gplusminus}
	\end{figure*}
	
	\begin{proposition}\label{prop:bitangent_upper_bound}
		For any pair $K_l$, $K_j$ of distinct obstacles, there are at most 4 undirected smooth (i.e. non-reflecting) geodesics which are tangent to both $K_l$ and $K_j$.
	\end{proposition}
	\begin{proof}
		First we begin with some definitions. Let $\gamma_1,\gamma_2:[0,1]\to S$ be two distinct, directed smooth geodesics in $S$, with $x_i = \gamma_i(0)\in\partial S$ and $y_i = \gamma_i(1)\in\partial S$, for $i=1,2$.
		Now suppose that $\gamma_1$ and $\gamma_2$ are both tangent to the same obstacle $K_l$, at two distinct points $p_1,p_2\in \partial K_l$ respectively. Parameterise $\partial K_l$ in the anti-clockwise direction as $k_l:\mathbb{S}^1\to S$. Let $E(s,t) = \flow_t(k_l(s),k_l'(s))$. Then there is a smooth positive function $t(s)>0$ such that $F(s) = \pi_1\circ E(s,t(s))$ is a smooth diffeomorphism onto $\partial S$. For $\widetilde E(s,t) = \flow_{-t}(k_l(s),k_l'(s))$, there is a corresponding smooth positive function $\widetilde t(s)>0$ such that $\widetilde F(s) = \pi_1\circ\widetilde E(s,\widetilde t(s))$ is a diffeomorphism on $\partial S$. Note that $F$ and $\widetilde F$ are only well-defined when both $K_j$ and $S$ are strictly convex. There is some $s_1\in\mathbb{S}^1$ such that $k_l(s_1)=p_1$. Suppose that $\langle\dot\gamma_1,\dot k_l(s_1)\rangle>0$, otherwise we can replace $\gamma_1(t)$ with $\gamma_1(1-t)$. Then $F(s_1) = y_1\in\partial S$. Now, there is some $a_0\in\mathbb{S}^1$ such that $F(a_0)=x_1$. Let $\alpha:[0,1]\to S$ be the smooth geodesic such that $\alpha(0) = x_1$ and $\alpha$ is tangent to $\partial K_l$ at $k_l(a_0)$. Consider the path components, $S_1$ and $S_2$ of $\partial S\backslash\{x_1,y_1\}$. Let $S_1$ be the path component containing $\alpha(1)$.
		Take a normal coordinate chart $\phi: M\supseteq U\to V$ about $x_1$ large enough so that $S\subseteq U$. Then $\phi(\gamma_1([0,1]))$ is the straight line in $V$ from $\phi(x_1)$ to $\phi(y_1)$. By \Cref{lemma:twocomponents}, $\phi(S)$ is therefore split by $\phi(\gamma_1)$ into two path-components $\phi(\tilde S_1)$ and $\phi(\tilde S_2)$. Their boundaries are precisely $\partial\tilde S_i = S_i\cup\gamma_1([0,1])$, for $i=1,2$. It follows that $\alpha((0,1])\subseteq \tilde S_1$, since $\alpha$ and $\gamma_1$ can intersect at most once. Since $\gamma_1$ is tangent to $\partial K_l$, and $\partial K_l$ is strictly convex, $\gamma_1$ intersects $\partial K_l$ exactly once. Thus we must have $\partial K_l\subseteq \overline{\tilde S_i}$ for either $i = 1$ or $i = 2$. But $\alpha$ is also tangent to $\partial K_l$, and $\alpha([0,1])\subseteq \overline{\tilde S_1}$, therefore $\partial K_l\subseteq \overline{\tilde S_1}$ (\Cref{fig:components}). 
		For each $s\in\mathbb{S}^1$ let $\widetilde G_s(t) = \pi_1\circ E(s,t)$. Let $G_s:\mathbb{S}^1\times [0,1]$ be a re-parameterisation of $\widetilde G_s$ such that $G(0) = \widetilde G_s(-\widetilde t(s))$, $G(1) = \widetilde G_s(t(s))$, and $\pdiff{t}\widetilde G_s(t(s)) = \lambda_s\pdiff{t}G_s(1)$ for some $\lambda_s>0$. For each $s\in\mathbb{S}^1$ there is some $\omega_s\in[0,1]$ such that $G_s(\omega_s)=\widetilde G_s(0)\in\partial K_l$. Then by construction, $G_s([\omega_s,1])\subseteq\tilde S_1$ for all $s\in F^{-1}(S_1)$. On the other hand, $G_s([\omega_s,1])\cap\tilde S_i\neq\emptyset$ for $i=1$ and $2$, for all $s\in F^{-1}(S_2)$. That is, $G_s$ intersects $\gamma_1$ exactly once, for all $s\in F^{-1}(S_2)$. Similarly $G_s([0,\omega_s]))\cap\tilde S_2\neq\emptyset$ for all $s\in\widetilde F^{-1}(S_2)$.
		Define the sets
		\[G^+(\gamma_1) = \{G_s : s\in F^{-1}(S_2)\}\]\[G^-(\gamma_1) = \{G_s : s\in\widetilde F^{-1}(S_2)\},\]
		and their union
		\[ G(\gamma_1) = \{G_s : s\in F^{-1}(S_2)\cup\widetilde F^{-1}(S_2)\}.\]
		This is precisely the set of all geodesics tangent to $\partial K_l$ with the same orientation as $\gamma_1$, which intersect $\gamma_1$ in the interior of $S$. We will show later that the geodesics in $G^+(\gamma_1)$ differ to those in $G^-(\gamma_1)$ in terms of where they intersect $\gamma_1$.
		
		Now suppose that $\gamma_1$ and $\gamma_2$ are both tangent to $\partial K_l$ and $\partial K_j$, such that both $\gamma_1$ and $\gamma_2$ intersect $\partial K_l$ before intersecting $\partial K_j$. Let $k_j:\mathbb{S}^1\to S$ be an anti-clockwise parameterisation of $\partial K_j$. For $i=1,2$ let $\gamma_i(t_i)=k_l(s_i)$ and $\gamma_i(t'_i)=k_j(s'_i)$ be the tangential intersection points of $\gamma_i$ with $\partial K_l$ and $\partial K_j$ respectively, for some $t_i,t'_i\in\reals$ and $s_i,s'_i\in\mathbb{S}^1$. Then $0<t_i<t'_i<1$ by assumption. Let $r_i^{(l)}$ and $r_i^{(j)}$ be the signs of $\langle\dot\gamma_i(t_i),k_l'(s_i)\rangle$ and $\langle\dot\gamma_i(t'_i),k_j'(s'_i)\rangle$ respectively. We claim that if $r_1^{(l)} = r_2^{(l)}$ then $r_1^{(j)}\neq r_2^{(j)}$. It suffices to show that this is true for $r_1^{(l)}=1$, since the alternative case follows by replacing $\gamma_i(t)$ with $\gamma_i(1-t)$. Assume that $r_1^{(j)}= r_2^{(j)}$.
		We begin by defining analogous maps and sets for $\partial K_j$ as follows. Let $E'(s,t) = \flow_{t}(k_j(s),r_1^{(j)}k'_j(s))$, and $\widetilde E'(s,t) = \flow_{-t}(k_j(s),r_1^{(j)}k'_j(s))$. Then as before, there are positive maps $\tau(s),\widetilde\tau(s)>0$ such that $F'(s) = \pi_1\circ E'(s,\tau(s))$ and $\widetilde F'(s) = \pi_1\circ\widetilde E'(s,\widetilde\tau(s))$ are diffeomorphisms onto $\partial S$. For each $s\in\mathbb{S}^1$ let $\widetilde H_s(t) = \pi_1\circ E'(s,t)$. Let $H_s:\mathbb{S}^1\times [0,1]$ be a re-parameterisation of $\widetilde H_s$ such that $H(0) = \widetilde H_s(-\widetilde \tau(s))$, $H(1) = \widetilde H_s(\tau(s))$, and $\pdiff{t}\widetilde H_s(\tau(s)) = \lambda'_s\pdiff{t}G_s(1)$ for some $\lambda'_s>0$. For each $s\in\mathbb{S}^1$ there is some $\mu_s\in[0,1]$ such that $H_s(\mu_s)=\widetilde H_s(0)\in\partial K_j$. Finally, define the sets
		\[H^+(\gamma_1) = \{H_s : s\in F'^{-1}(S_2)\}\]\[H^-(\gamma_1) = \{H_s : s\in\widetilde F'^{-1}(S_2)\},\]
		and their union
		\[ H(\gamma_1) = \{H_s : s\in F'^{-1}(S_2)\cup\widetilde F'^{-1}(S_2)\}.\]
		We now have two cases, depending on the sign of $r_1^{(j)}$. First, if $r_1^{(j)}<0$ (\Cref{fig:gplusminus_negative}), then $\partial K_j\subseteq\overline{\tilde S_2}$, since both $\partial K_l$ and $\partial K_j$ were given anti-clockwise parameterisations. Now since $\gamma_2$ is tangent to both $\partial K_l\subseteq\overline{\tilde S_1}$ and $\partial K_j\subseteq\overline{\tilde S_2}$, it must intersect $\gamma_1$. Therefore $\gamma_2 = G_{s_2} = H_{{s'_2}}\in G(\gamma_1)\cap H(\gamma_1)$. Note that this is only true due to the inclusion of the sign term $r_1^{(j)}$ in the definition of $E'$ and $\widetilde E'$, which ensures that $G_{s_2}$ and $H_{s'_2}$ have the same orientation as $\gamma_2$. Now since $\gamma_2$ intersects $\partial K_l$ prior to $\partial K_j$, it follows that $G_{s_2}([\omega_s,1])\cap \widetilde S_2\neq\emptyset$. Hence $G_{s_2}\in G^+(\gamma_1)$. Reasoning in a similar manner, $H_{s'_2}([0,\mu_s])\cap \tilde S_1\neq\emptyset$ so $H_{s'_2}\in H^+(\gamma_1)$. Let $z:G(\gamma_1)\to\gamma_1([0,1])$ be the intersection between $G_{s^*}$ and $\gamma_1$, for any $G_{s^*}\in G(\gamma_1)$. We claim that, $z(G^+(\gamma_1))\subseteq\gamma_1([0,t_1])$. 
		
		Consider $\partial G^+(\gamma_1)$. There are two smooth geodesics in $\partial G^+(\gamma_1)$, one of which is $\gamma_1$. The other is $G_{\widetilde s}$, where $\widetilde s = F^{-1}(x_1)$. We remark that therefore $\partial G^+(\gamma_1)\cap G(\gamma_1) = \emptyset$. Now for any $G_{s^*}\in G^+(\gamma_1)$, the point of tangency $G_{s^*}(\omega_{s^*})\in\partial K_l$ is contained in the region bounded by $k_l([\widetilde s, s_1])$, $\gamma_1([0,t_1])$ and $G_{\widetilde s}([\omega_{\widetilde s},1])$. Furthermore, $G_{s^*}(1)\in S_2$, so $G_{s^*}$ must intersect either $\gamma_1([0,t_1])$ (in which case the claim is true), or $G_{\widetilde s}([\omega_{\widetilde s},1])$. Suppose the latter case holds. $G_{\widetilde s}$ splits $\overline{\widetilde S_1}\backslash G_{\widetilde s}([0,1])$ into two path components, $S_1^A$ and $S_1^B$ (by \Cref{lemma:twocomponents}). Let $\overline {S_1^A}$ be the component containing $K_l$. Then for some $(a,b)\subseteq [\omega_{s^*},1]$, we have $G_{s^*}((a,b))\subseteq S_1^B$. But $\partial S_1^B\cap\partial\widetilde S_2 = \{x_1\}$, so $G_{s^*}$ cannot reach $\widetilde S_2$ without intersecting $S_1^A$. That is, $G_{s^*}([b,1])\cap S_1^A\neq\emptyset$, since $G_{s^*}(1)\in S_2$. Therefore $G_{s^*}$ must intersect $G_{\widetilde s}$ twice, a contradiction, which shows our claim holds. Note that a similar argument will show that $z(G^-(\gamma_1))\subseteq \gamma_1([t_1,1])$.
		
		Hence $G_{s_2}$ intersects $\gamma_1$ prior to the first tangency $\gamma_1(t_1)$. Similarly, if $\tilde z:H(\gamma_1)\to\gamma_1([0,1])$ is the intersection between $H_{s^*}$ and $\gamma_1$ for any $H_{s^*}\in G(\gamma_1)$, then \[\tilde z(H^{+}(\gamma_1))\subseteq\gamma_1([t'_1,1])\text{ and }\tilde z(H^{-}(\gamma_1))\subseteq\gamma_1([0,t'_1]).\]
		Note that if $r_1^{(j)}>0$ then $H^+(\gamma_1)$ and $H^-(\gamma_1)$ swap. Recalling that $\gamma_2 = G_{s_2} = H_{s'_2}\in G^+(\gamma_1)\cap H^+(\gamma_1)$, we have $z(\gamma_2)\in \gamma_1([0,t_1])$, while $\widetilde z(\gamma_2) \in \gamma_1([t'_1,1])$. That is, $z(\gamma_2)\neq\widetilde z(\gamma_2)$, meaning that $\gamma_2$ must intersect $\gamma_1$ twice, a contradiction.
		
		 We now consider the case where $r_1^{(j)}>0$ (\Cref{fig:gplusminus_positive}). In this case, both $\partial K_l$ and $\partial K_j$ are in the same connected component $\overline{\tilde S_1}$. Suppose that $\gamma_2$ intersects $\gamma_1$, then as in the previous case we have $\gamma_2 = G_{s_2} = H_{s'_2}\in G(\gamma_1)\cap H(\gamma_1)$. Also let $z$ and $\tilde z$ be defined in the same way as in the previous case. Since $\partial K_j$ is in the same component of $S$ as $\partial K_l$, and $\gamma_2$ can intersect $\gamma_1$ at most once, it follows that the intersection must occur after $\gamma_2$ is tangent to $\partial K_j$. That is, $z(G_{s_2})=\tilde z(H_{s'_2})\in\gamma_2([t'_2,1])$. Note that by construction, $z(G_{s^*})\in G_{s^*}([0,\omega_s])$ for any $G_{s^*}\in G^-(\gamma_1)$. Therefore, $G_s\in G^+(\gamma_1)$, since $z(G_s)\not\in G_s([0,\omega_s])$. Now $G_s\in G^+(\gamma_1)$, as in the previous case, implies that $z(G_s)\in\gamma_1([0,t_1])$. Therefore it follows that $z(G_s)\in\gamma_1([0,t_1])\cap\gamma_2([t'_2,1])$. Pick $t^*_1,t^*_2\in[0,1]$ such that $\gamma_1(t^*_1)=\gamma_2(t^*_2)$ is the point of intersection between $\gamma_1$ and $\gamma_2$.
		 Then it follows that $\partial K_j$ must be contained in the region bounded by $\gamma_1([t^*_1,t_1]), \gamma_2([t'_2,t_2^*])$ and $k_l([s_1,s_2])$ (\Cref{fig:bitangent_enclosed_region}). But then $t'_1\in [t^*_1,t_1]$, a contradiction since we assumed that $t'_1>t_1$. Therefore $\gamma_2$ cannot intersect $\gamma_1$. Finally, suppose that $\gamma_2$ does not intersect $\gamma_1$. Then $\gamma_2\not\in G(\gamma_1)\cup H(\gamma_1)$. Note that by assumption $t'_2>t_2$, hence $\partial K_j$ must be contained within the region bounded by $\gamma_2([t'_2,1]), \gamma_1([0,t_1,]), k_l([s_1,s_2])$ and $\partial S$. But then $t'_1\in [0,t_1]$, once again a contradiction since we assumed that $t'_1>t_1$. Thus our claim that $r_1^{(j)}\neq r_2^{(j)}$ whenever $r_1^{(l)} = r_2^{(l)}$ holds is proved.
		
		Therefore there can be at most eight directed smooth geodesics which are tangent to $\partial K_l$ and $\partial K_j$ depending on the pairs of signs $r_1^{(l)}$ and $r_1^{(j)}$, and depending on which obstacle they intersect first. That is, there are at most 4 undirected smooth geodesics which are tangent to both $\partial K_l$ and $\partial K_j$.
	\end{proof}
	
	\begin{figure}
		\center
		\begin{tikzpicture}
			\tikzmath{
			    \r1 = 2cm;%radius of smaller circle
			    \d1 = 4cm; %diameter of smaller circle
			    coordinate \c;
			    \c1 = (0,0);
			}
			
			\node[circle,draw] (c1) at (\c1) [minimum size = \d1]{$K_l$};
			\coordinate (c1a) at ($(\c1)+(2,0)$);
			\coordinate (c1b) at ($(\c1)+(-2,0)$);
			
			\coordinate (a) at (6,2);
			\coordinate (b) at (6,4);
			\coordinate (ta1) at (tangent cs:node=c1,point={(a)},solution=1);
			\coordinate (ta2) at (tangent cs:node=c1,point={(a)},solution=2);
			\coordinate (tb1) at (tangent cs:node=c1,point={(b)},solution=1);
			\coordinate (tb2) at (tangent cs:node=c1,point={(b)},solution=2);
			
			\draw[name path = gamma1, -{Latex[length=5mm,width=2mm]}, shorten >= -1cm] (a) -- (ta1) node[midway, above]{$\gamma_1$};
			
			\draw[name path = gamma2, -{Latex[length=5mm,width=2mm]}, shorten <= -1cm] (tb2) -- (b) node[midway, below right]{$\gamma_2$};
			
			\fill[name intersections = {of = gamma1 and gamma2, by={z}}] (z) circle (0.1cm) node[below right]{$\gamma_1(t^*_1)=\gamma_2(t^*_2)$};
			
			\draw [pattern=north west lines, pattern color=gray, opacity = 0.4] (z) -- (ta1) -- (tb2) -- cycle;
			\node[circle,fill,white] at (\c1) [minimum size = \d1]{$K_l$};
			\node at (\c1){$K_l$};

			\fill (ta1) circle (0.1cm) node[above]{$\gamma_1(t_1)$};
			\fill (tb2) circle (0.1cm) node[below,right]{$\gamma_2(t_2)$};
			
			\pic [draw, angle radius=2cm,->] {angle=c1a--c1--c1b};
			\pic [draw, angle radius=2cm,->] {angle=c1b--c1--c1a};
			\pic [draw, blue, angle radius=2cm, {Parenthesis}-{Parenthesis},thick] {angle=ta2--c1--ta1};
			\node[blue] (g+) at (1.3,0.6) {$G^+$};
		\end{tikzpicture}
		\caption{When $\gamma_2\in G^+(\gamma_1)$, it forms an enclosed region together with $\gamma_1$ and $\partial K_l$ (shaded). Assuming that $r^{(j)}_1,r^{(j)}_2>0$ forces $K_j$ to lie within the shaded region.}\label{fig:bitangent_enclosed_region}
	\end{figure}
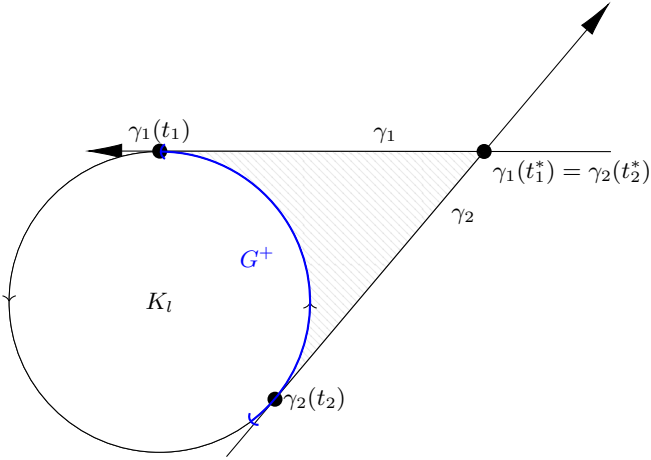
	
	\begin{proposition}\label{prop:bitangent_lower_bound}
		Suppose that $n\geq 3$. Let $K_l$ and $K_j$ be two distinct obstacles. Then there are at least 4 undirected smooth geodesics which are tangent to both $K_l$ and $K_j$.
	\end{proposition}
	\begin{proof}
		First, we show that there exists at least one smooth geodesic tangent to $\partial K_l$ which does not intersect $\partial K_j$. Let $K_m$ be any third obstacle distinct from both $K_l$ and $K_j$. Pick any point $q\in \partial K_m$ and take a normal coordinate chart $\psi: U\to V\subseteq\reals^2$ about $q$ such that $S\subseteq U$. Then there is some straight line $L\subset V$ intersecting $\psi(\partial K_m)$ at $\psi(q)$, which is tangent to $\psi(\partial K_l)$. This line corresponds to a smooth geodesic $\gamma = \psi^{-1}(L)$ through $q$ which is tangent to $\partial K_l$ in $U$. Note that $\gamma$ may intersect $\partial K_m$ more than once. Since our obstacles are in general position (by assumption), it follows that $\gamma$ cannot intersect any other obstacle apart from $K_l$ and $K_m$. Thus $\gamma$ is a smooth geodesic tangent to $\partial K_l$ which does not intersect $\partial K_j$. Note that by the same argument there is some smooth geodesic $\gamma'$ which intersects $\partial K_j$ and is tangent to $\partial K_l$.
		
		Parameterise $\partial K_l$ in an anti-clockwise fashion as $k_l(s):\mathbb{S}^1\to\partial K_l$. Define the function $E:T\partial K_l\times\reals\to TS$ by $(s,t)\mapsto \flow_t(k_l(s),k_l'(s))$.
		By the argument in the previous paragraph, there is some $q\in\partial K_j$ and a smooth geodesic $\gamma$ through $q$ which is tangent to $\partial K_l$. By the implicit function theorem there are open sets $W\subseteq\partial K_l$ and $Z\subseteq\partial K_j$ such that $q\in Z$, along with a unique diffeomorphism $\phi:W\to Z$ such that the smooth geodesic starting  at $k_l(s)=w\in W$ in the direction $k_l'(s)$ first intersects $\partial K_j$ at $\phi(w)$. That is, if $t(s)>0$ is the distance between $w$ and $\phi(w)$ then $\phi(w) = \pi_1\circ E(s,t(s))$. Furthermore the intersections of the geodesics $F_s(t) = \pi_1\circ E(s,t)$ with $\partial K_l$ at $\phi(k(s))$ are transversal. Provided that $F_s(t)$ remains transversal to $\partial K$ for all $s\in k_l^{-1}(W)$, we can extend $\phi$ by expanding $Z$ and hence $W$ (see the proof of \Cref{lemma:singlereflection} for this construction). Suppose that there are no smooth bitangent geodesics which are tangent to both $\partial K_l$ and $\partial K_j$. It follows that $Z = \partial K_j$, and since $W$ is diffeomorphic to $Z$, it follows that $W = \partial K_l$. But then for every point $w\in\partial K_l$ the geodesic $F_{k_l^{-1}(w)}$, tangent to $\partial K_l$ will intersect $\partial K_j$ at $\phi(w)$ (transversally) by construction. This contradicts our argument above, since there must be some smooth geodesic tangent to $K_l$ which does not intersect $K_j$. Hence there must be at least one bitangent smooth geodesic, tangent to both $\partial K_l$ and $\partial K_j$.
		
		Suppose that there is only one bitangent smooth geodesic, $\alpha^*:[0,1]\to S$, tangent to both $\partial K_l$ and $\partial K_j$. Let $p\in\partial K_j$ be the point of tangency of $\alpha^*$, it follows that $Z = \partial K_j\backslash\{p\}$. Parameterise $W$ in the anti-clockwise direction as $w:(0,1)\to W$. Let $w^* = \lim\limits_{s\to 1}w(s)$, then $p = \lim\limits_{w\to w^*}\phi(w)$. Note that since there is only one point of tangency, we must also have $\lim\limits_{s\to 0}w(s) = w^*$. Now taking a normal coordinate chart $\psi$ as before, centred about $w^*$ we have that $\psi(\alpha^*([0,1]))$ is a straight line, which we may extend in both directions until intersecting $\psi(S)$ at both ends. Denote this extended smooth geodesic by $\widetilde{\alpha^*}:[a,b]\to S$. Then by \Cref{lemma:twocomponents}, $S\backslash\widetilde{\alpha^*}([a,b])$ has two connected components. Denote these two components by $\psi(\tilde S_1)$ and $\psi(\tilde S_2)$. Suppose that $K_l$ and $K_j$ are both in the same component $\tilde S_1$. 
		Re-parameterise $\widetilde{\alpha^*}$ as $\gamma_1:[0,1]\to S$, so that $\pdiff{s}\widetilde{\alpha^*}(a) = \lambda\dot\gamma_1(0)$ for some $\lambda > 0$. Then using the same notation as in \Cref{prop:bitangent_upper_bound}, consider the set $G^+(\gamma_1)$ (\Cref{fig:gplusminus_positive}). Recall that for each $s\in\mathbb{S}^1$ there is some $\omega_s>0$ such that $G_s(\omega_s)\in\partial K_l$ is the point of tangency of $G_s$ with $\partial K_l$.
		Also let $t'_1>t_1>0$ be such that $\gamma_1(t_1)\in\partial K_l$ and $\gamma_1(t'_1)\in\partial K_l$ are the points of tangency of $\gamma_1$ with $\partial K_l$ and $\partial K_j$ respectively. As shown in \Cref{prop:bitangent_upper_bound}, for any $G_s\in G^+(\gamma_1)$, the intersection $z(G_s)$ between $G_s$ and $\gamma_1$ must lie in $\gamma_1([0,t_1])$. Now since $Z = \partial K_j\backslash\{p\}$, it follows that for every $G_s\in G^+(\gamma_1)$, the segment $G_s([\omega_s,1])$ intersects $\partial K_j$ transversally. Let $1>\sigma_s>\theta_s>\omega_s$ be such that $G_s(\sigma_s)\in\partial K_j$ and $G_s(\theta_s)=z(G_s)$ are the points of intersection of $G_s$ with $\partial K_j$ and $\gamma_1$ respectively. Then $G_s([\omega_s,\theta_s))\subseteq \tilde S_1$, while $G_s((\theta_s,\sigma_s])\cap\tilde S_i\neq\emptyset$ for $i=1,2$. It follows that $G_s(\theta_s,\sigma_s])$ must also intersect $\gamma_1([0,1])$. Thus $G_s$ intersects $\gamma_1$ more than once, leading to a contradiction. 
		Now suppose that $K_l$ and $K_j$ are in two separate components, $\partial K_l\subseteq\overline{\tilde S_1}$ and $\partial K_j\subseteq\overline{\tilde S_2}$. Once again using the same notation as \Cref{prop:bitangent_upper_bound}, consider the set $G^-(\gamma_1)$ (\Cref{fig:gplusminus_negative}). Let $G_s(\omega_s)\in\partial K_l$, $G_s(\sigma_s)\in\partial K_j$ and $G_s(\theta_s)=z(G_s)$ denote the same points as in the previous case. Then by construction, $G_s([\omega_s,1])\cap S_i\neq\emptyset$ for $i=1,2$, since $\omega_s<\sigma_s$ and $\partial K_l$ and $\partial K_j$ are in different components. However, as shown in \Cref{prop:bitangent_upper_bound}, $z(G_s)\in G_s([0,\omega_s])$. That is, $0<\theta_s<\omega_s<\sigma_s$. So once again, $G_s$ must intersect $\gamma_1$ at two distinct points, leading to a contradiction.
		
		Hence there must be at least two bitangent smooth geodesics when parameterising in the anti-clockwise direction. Taking the parameterisation in the clockwise direction we find at least two  bitangents once again. We claim that the pairs of bitangents must be distinct. Suppose that $\gamma^+(t)$ is a bitangent smooth geodesic starting at $k_l(s)\in\partial K_l$ with initial direction $k_l'(s)$. Let $\gamma^{-}(t)$ be the smooth geodesic starting at $k_l(s)$ with initial direction $-k_l'(s)$. Define $\gamma(t) = \gamma^+(t)$ for $t\geq 0$ and $\gamma(t)=\gamma^{-}(-t)$ for $t<0$. Now since $K_j$ is strictly convex, and $\gamma(t)$ is tangent to $\partial K_j$ at $\phi(k_l(s))$, the smooth geodesic $\gamma(t)$ must intersect $\partial K_j$ exactly once. Since $\gamma(t)$ cannot self intersect, it follows that $\gamma^{-}(t)$ is not tangent to $\partial K_j$, and in fact does not intersect $\partial K_j$ anywhere. Hence the four bitangent smooth geodesics must be distinct.
	\end{proof}
	
	Let $\sT_i^j\subset \sT$ be the set of travelling times generated by geodesics which reflect off $\partial K$ exactly $i$ times and are tangent to exactly $j$ connected components of $K$. Since we have assumed that $K$ is in general position, it follows that $\sT_i^j = \emptyset$ for all $j\geq 3$. Consider the set of travelling times $\sT_0$ which are generated by smooth geodesics. Then we know $\sT_0 = \cup_{j=0}^{2}\sT^j_0$.
	
	\begin{corollary}\label{lemma:number_of_obstacles}
		Suppose that $n\geq 3$. The set $\sT_0^2$ contains exactly $4n(n-1)$ discrete points, and $\sT_0^1$ is the union of $4n(n-1)$ open arcs.
	\end{corollary}
	\begin{proof}	
		Each obstacle has exactly $4(n-1)$ geodesics which are tangent to both it and another obstacle. In total this gives $4n(n-1)$ such (directed) geodesics. Furthermore, each obstacle has $4(n-1)$ points of double-tangency. Every pair of successive points determines an open arc of points along the boundary of the obstacle which generate geodesics tangent only to that obstacle. Every such arc determines an open arc of travelling times, and the union of these open arcs is exactly $\sT_0^1$. Therefore $\sT_0^1$ is the union of $4n(n-1)$ disjoint open arcs.
	\end{proof}
	\Cref{lemma:number_of_obstacles} allows us to determine the number of obstacles $n$ directly from the travelling times in a rather practical manner. Computing $\sT_0^2$ requires a minimal amount of data in comparison to computing the sets $\sT_i^1$ which are used to recover $K$. The amount of data required to identify $\sT_i^1$ increases by an order of magnitude for each $i\geq 0$ (see \Cref{example:echograph}). Note that in the case where there are only two obstacles ($n=2$), only the upper bound given in \Cref{prop:bitangent_upper_bound} holds. However, this is sufficient to conclude there are only two obstacles, since there will be at most 8 points in $\sT_0^2$.
	
	\section{The structure of the set of travelling times $\sT$}
	
	We define the map \[I:\T{S_K}\to \T{S_K}\]
	\begin{equation}
 		I(x,\omega) = (x,\omega-2\left<\omega,v_x \right> v_x)
 	\end{equation}
	where $v_x$ is the unit normal to $\partial S_K$ at $x$. This is the inversion map by symmetry to the normal to the boundary. Note that $I$ is a diffeomorphism.	
	
	\begin{lemma}\label{lemma:multireflection}
		Let $\gamma$ be a generalised geodesic in $S_K$ generated by $(x_0,\omega_0)\in \T{S_K}$ with successive reflection points $x_1,\dots,x_n$ on $\partial S_K$. Then there are neighbourhoods $W$ of $(x_0,\omega_0)\in \T{S_K}$, and $V_i$ of $x_i$ in $\partial S_K$ and unique smooth maps \[{x_i(x,\omega):W\to V_i}\] such that any generalised geodesic generated by $(x,\omega)\in W$ will have successive reflection points $x_i(x,\omega)$.
	\end{lemma}
	\begin{proof}
		By \Cref{lemma:singlereflection} there exist a neighbourhood $W$ of $(x_0,\omega_0)$ and a map $\tau_1$ from $W$ to $\reals^+$ such that $\flow_{\tau_1(x,\omega)}$ is a diffeomorphism from $W$ onto a neighbourhood $U_1$ of $(x_1,\omega_1)$ along the boundary $\partial K$.	 Now by \Cref{lemma:singlereflection}, for each $i=1,\dots,n$ there is a neighbourhood ${\widetilde U_{i-1}\subseteq I(U_{i-1})}$ and a map $\tau_i:\widetilde U_{i-1}\to \reals^+$ such that $\flow_{\tau_i(x,\omega)}$ is a diffeomorphism from $\widetilde U_{i-1}$ onto a neighbourhood $U_i$ of $(x_i,\omega_i)$ along the boundary $\partial K$. We shrink each $U_i$ so that $I(U_i) = \widetilde U_i$, and shrink $W$ so that $\flow_{\tau_1(x,\omega)}(W)=U_1$.
		Now define maps $X_i:W\to U_i$ recursively as follows:
		\begin{align}
	 		&X_1(x,\omega) = \flow_{\tau_1(x,\omega)}(x,\omega)\\
			&X_i(x,\omega) = \flow_{\tau_i(I\circ X_{i-1}(x,\omega))}(I\circ X_{i-1}(x,\omega))\label{eq:reCref}
		\end{align}
		Note that each $X_i$ is a diffeomorphism. Finally, the desired maps are 
		\begin{equation*}
			x_i = \pi_1\circ X_i:W\to\pi_1(U_i)\qedhere
		\end{equation*}
	\end{proof}
	We will use \Cref{thm:timeangle,lemma:distinct_times,thm:convexlemma,thm:confrontorth} often, see \cite{GNS2020Published} for their proofs:
	\begin{lemma}\label{thm:timeangle}
		Suppose that $\gamma$ is a non-trapped generalised geodesic in $S_K$ from $x\in \partial S$ to $y\in\partial S$. Then $grad_x T = -\dot\gamma(t_0)/\norm{\dot\gamma(t_0)}$, where $T(x,y)$ is the length of the geodesic $\gamma$.
	\end{lemma}
	\begin{lemma}\label{lemma:distinct_times}
		Fix $x_0\in \partial S$. The set of pairs of distinct directions $\omega_1,\omega_2\in \T{S}_{x_0}$ which generate generalised geodesics with  the same endpointst and the same travelling time is countable.
	\end{lemma}
	\begin{lemma}
	\label{thm:convexlemma}
		Suppose $c:[a,b]\to M$ is a smooth, unit speed, strictly convex curve. For each $u_0\in [a,b]$ there exists  a smooth, strictly convex curve $y$ on a neighbourhood of $u_0$ such that $\pdiff{u}{y}(u)$ is orthogonal to the parallel translate of $\pdiff{u}{c}(u)$ along the geodesic from $c(u)$ in the direction $\pdiff{u}{c}(u)$. 
	\end{lemma}
	\begin{lemma}\label{thm:confrontorth}
		Let $\gamma$ be a generalised geodesic in $S_K$. Suppose there are two convex fronts, $X$ and $Y$ such that $\dot\gamma(0)$ is the outward unit normal to $X$ and for some $t_0>0$, the velocity $\dot\gamma(t_0)$ is the inward unit normal to $Y$. Also suppose that $\gamma$ reflects transversally between $X$ and $Y$. Parameterise $X$ as \[x:[a,b]\to S_K\] with $x(u_0)=\gamma(0)$ and unit outward normal $\omega(u)$. Then there exists an open set $U\subseteq [a,b]$ containing $u_0$ such that $(x(u_0),\omega(u_0))$ generates a geodesic that hits $Y$ orthogonally, and $(x(u),\omega(u))$ does not, for all $u\in U\backslash\{u_0\}$.
	\end{lemma}
	
	Let $\gamma_{(x,\omega)}$ be the geodesic such that $\gamma_{(x,\omega)}(0)=x$ and $\dot\gamma_{(x,\omega)}(0)=\omega$, for $(x,\omega)\in \T{S}\backslash Trap(S_K)$. Let $t(x,\omega)$ be the travelling time of $\gamma_{(x,\omega)}$.  Define the endpoint map as follows:
	 \begin{equation}\label{eq:forwardmap}
	 	\mathcal{P}(x,\omega) = (\gamma_{(x,\omega)}(t(x,\omega)),\dot\gamma_{(x,\omega)}(t(x,\omega)))\in \T{S}
	 \end{equation}
	Note that if $\gamma_{(x_0,\omega_0)}$ is nowhere tangent to $\partial K$, then the restriction of $\mathcal{P}$ to a neighbourhood of $(x_0,\omega_0)$ is a diffeomorphism by \Cref{lemma:multireflection}.
	
	\begin{lemma}\label{lemma:travelling_time_function}
		Suppose that $\gamma_{(x_0,\omega_0)}$ is nowhere tangent to $\partial K$, and denote $(x',\omega') = \mathcal{P}(x_0,\omega_0)$. Then there exists a neighbourhood $U$ of $x_0$ in $\partial S$ and a map $\tau : U\to \reals$ such that for all $x\in U$, we have $(x,x',\tau(x))\in\sT$ and \[\pi_1\circ\mathcal{P}(x,-\nabla\tau(x)) = x'\]
	\end{lemma}
	\begin{proof}
		We begin by keeping all the definitions as in the proof of \Cref{lemma:multireflection}. Suppose $\gamma_{(x_0,\omega_0)}$ reflects $k$ times, then $x' = x_{k+1}(x_0,\omega_0)$. We look at the final smooth geodesic section, restricting $\pi_1\circ\flow_t$ to $\widetilde U_k$. Note that $\pi_1\circ\flow_t$ is a submersion, so taking $\mathcal{U}=(\pi_1\circ\flow_t)^{-1}(x')$ gives a codimension 1 submanifold of $\widetilde U_k$. Thus ${X_k}^{-1}(\mathcal{U})$ is a codimension 1 submanifold of $W$. Note that the set of directions which give $x_{k+1}(x,\omega)=x'$ for each fixed $x\in W$ is countable and discrete (by \Cref{lemma:distinct_times}). Hence we can shrink ${X_k}^{-1}(\mathcal{U})$ around $(x_0,\omega_0)$ such that if both $(x,\omega), (x,\omega')\in{X_k}^{-1}(\mathcal{U})$ then $\omega = \omega'$. We can now define a function $\phi:\pi_1({X_k}^{-1}(\mathcal{U}))\to\pi_2({X_k}^{-1}(\mathcal{U}))$ by letting $\phi(x)$ be the unique direction such that $(x,\phi(x))\in {X_k}^{-1}(\mathcal{U})$. Also let $\tau : \pi_1({X_k}^{-1}(\mathcal{U}))\to\reals$ be the function defined by setting $\tau(x)$ as the unique travelling time determined by $(x,\phi(x))$, so that $(x,x',\tau(x))\in \sT$. Note that both $\phi$ and $\tau$ are smooth. It now follows from lemma 5 in \cite{GNS2020Published} that $\nabla\tau(x) = -\phi(x)$, so $\pi_1\circ\mathcal{P}(x,-\nabla\tau(x)) = x'$ and the proof is complete.
	\end{proof}
	
	We now define another set, a graph of the set of travelling times for every $x_1\in\partial S$,
	\[\sT^j_i(x_1) = \{(x,t):(x,x_1,t)\in\sT^j_i\}\]
	\[\sT^j(x_1) = \cup_{i=0}^{\infty}\sT^j_i(x_1)\]
	That is, the set of travelling times of geodesics which end at $x_1$, reflect exactly $i$ times and are tangent to $\partial K$ exactly $j$ times.
	The order, $o(\gamma)$ of a geodesic $\gamma$ is the number of intersections between $\gamma$ and $\partial K$. Let $d_K$ be the minimum distance between obstacles in $K$, and $t(\gamma)$ be the travelling time of $\gamma$. Then we have
	\begin{equation}
		o(\gamma)d_K\leq t(\gamma) \leq o(\gamma)diam(S)
	\end{equation}
	\begin{proposition}\label{lemma:open_arc_generators}
		For every $x_1\in\partial S$, the set $\sT^0(x_1)$ is a countable union of pairwise-transverse, smooth, bounded open arcs $\alpha_i$ in $\partial S\times\reals$. Furthermore, each arc $\alpha_i$ has a corresponding $\tau_i:U_i\to\reals$ such that:
		\begin{itemize}
			\item $U_i$ is an open arc in $\partial S$ and $x\mapsto (x,\tau_i(x))$ is a diffeomorphism from $U_i$ onto $\alpha_i$.
			\item For every $x\in U_i$ the geodesic generated by $(x,-\nabla\tau(x))$ is nowhere tangent to $\partial K$ and intersects $\partial S$ at $x_1$.
		\end{itemize}
	\end{proposition}
	\begin{proof}
		Take $(x_0,t_0)\in\sT^0(x_1)$, by \Cref{lemma:distinct_times} there are countably many $\omega_i$ such that $\gamma_{(x_0,\omega_i)}$ starts from $x_0$, ends at $x_1$ and has travelling time $t_0$. By \Cref{lemma:travelling_time_function} for every $\omega_i$ there exists a neighbourhood $U_i$ of $x_0$ in $\partial S$ and a map $\tau_i:U_i\to\reals$ such that for all $x\in U_i$ we have $\pi_1\circ\mathcal{P}(x,-\nabla\tau(x)) = x_1$. It now follows that the map $x\mapsto (x,\tau_i(x))$ is a diffeomorphism onto an open arc $\alpha_i$. Furthermore the arcs are pairwise-transverse at $x_0$, since $\nabla\tau_i(x_0) = \omega_i$. We may extend each $\alpha_i$ by applying \Cref{lemma:travelling_time_function} again at a point $(x',\tau_i(x'))\in\alpha_i$, with $x'\neq x_0$. There we get another set of maps $\tau_j:U_j\to\reals$, only one of which satisfies $\nabla\tau_{j'}(x') = \nabla\tau_i(x')$. By uniqueness of the maps $\tau$, the two must agree on $U_i\cap U_{j'}$ and so $\alpha_i$ is smoothy extended to $U_i\cup U_{j'}$. Therefore we can extend the $\alpha_i$'s to unique maximal open arcs which are pairwise tranverse, whose boundary is determined by geodesics tangent to $\partial K$. Finally, given $x,y\in \alpha_i$, the orders of the geodesics $\gamma_{(x,\tau_i(x))}$ and $\gamma_{(y,\tau_i(y))}$ must be equal by continuity. We can therefore write $o(\gamma_{(x,\tau_i(x))}) = o(\alpha_i)$ for any $x\in\alpha_i$. This gives another set of bounds:
		\begin{equation}\label{eq:arc_order_bound}
			o(\alpha_i)d_K\leq \inf_{x\in\alpha_i}\tau_i(x) \leq \sup_{x\in\alpha_i}\tau_i(x) \leq o(\alpha_i)diam(S)
		\end{equation}
		Consider the set of geodesics ending at $x_1$ which reflect exactly $k\geq 0$ times. By \Cref{thm:confrontorth}, finitely many of these geodesics are tangent to $\partial K$. Therefore there are finitely many maximal arcs $\alpha_i$ for each order $k$. Thus there are countably many maximal arcs in total, with $\sT^0(x_1)=\cup_{i\geq 1}\alpha_i$.
	\end{proof}
	Note that from \Cref{eq:arc_order_bound} we see that the closed arcs $\overline{\alpha_i}$ intersect at most finitely many other closed arcs $\overline{\alpha_{i'}}$, since at a point of intersection, the travelling times are equal, and hence bound the order. There are finitely many $\alpha_i$ of each order, thus bounding the number of possible intersections.
	
	\begin{corollary}
		$\sT^0(x_1)$ is open and dense in $\sT(x_1)$, therefore $\sT^1(x_1)\cup\sT^2(x_2) = \cup_{i\geq 1} \partial\alpha_i$
	\end{corollary}
	\begin{corollary}\label{lemma:travelling_times_dense}
		The set $\sT^1(x_1)$ is open and dense in $\sT^1(x_1)\cup\sT^2(x_1)$ and $\sT^2(x_1)$ is discrete. 
	\end{corollary}
	
	Now take some $(x_0,t)\in\partial \alpha_i$ and define $v_0 = \lim_{x\to x_0}-\nabla\tau_i(x)$. Then $\gamma_{(x_0,v_0)}$ will be a geodesic from $x_0$ to $x_1$ which is tangent to $\partial K$ at least once. The following results follows exactly as in \cite{math9192434}.
	\begin{proposition}\label{lemma:arc_intersection}
		Given $(x_0,t)\in\sT^1(x_1)$, there exist unique $i,i'\geq 1$ such that $\{(x_0,t)\} = \partial\alpha_i\cap\partial\alpha_{i'}$, with $o(\alpha_i)+1 = o(\alpha_{i'})$. For all $x\in U_i\cap U_{i'}$ we have $\tau_i(x) < \tau_{i'}(x)$. Note that both $U_i$ and $U_{i'}$ lie on the same side of $x_0$ in $\partial S$. Furthermore, \[\lim_{x\to x_0} \tau_i(x) = \lim_{x\to x_0} \tau_{i'}(x) = t \text{ and }\]\[ \lim_{x\to x_0} \nabla\tau_i(x) = \lim_{x\to x_0} \nabla\tau_{i'}(x).\]
	\end{proposition}
	
	\Cref{lemma:arc_intersection} has the following interesting consequence, since the arcs $U_i$ and $U_{i'}$ lie on the same side of $x_0$ in $\partial S$.
	\begin{corollary}
		$\sT^1(x_1)\cup\sT^2(x_1)$ is the closure of the set of all isolated cusps in $\sT(x_1)$
	\end{corollary}
	This implies that we can detect the points of tangency in the set of travelling times directly. One  could do this by observing the cusps in an appropriate embedding of the travelling times, as in the following example.
	
	\begin{example}\label{example:echograph}
		\Cref{fig:echograph} displays a so-called echograph of the set $\sT(x_1)$, up to 3 reflections, of obstacles in the Poincare half-plane. $x_1$ is the point in the bottom right of the outer boundary $\partial S$ where the two blue arcs (representing geodesics which do not intersect the obstacles) meet. The embedding $E:\sT(x_1)\to\reals^2$ is given by $(x_0,t_0)\mapsto x_0+t_0\eta(x_0)$, where $\eta(x_0)$ is the normal to $\partial S$ at $x_0$. Note that we have chosen to embed $\sT(x_1)$ in the Euclidean plane to allow for an easier interpretation of the arcs. One could also choose an analogous embedding in the Poincare half-plane, although the important features of the echograph (namely the cusps formed by the arcs) will remain all the same.
		
		Each time an arc meets another at a cusp, the order of the arcs must have a difference of exactly one. We denote the different orders by the use of colour, blue for the arcs of order 0 and red for those of order 3. Note that for increasing the order of an arc will increase the number of data points computed by an order of magnitude.
	\end{example}
	
	\begin{figure*}
		\center
		\includegraphics[width=\linewidth,trim={0px 0px 0px 20px},clip]{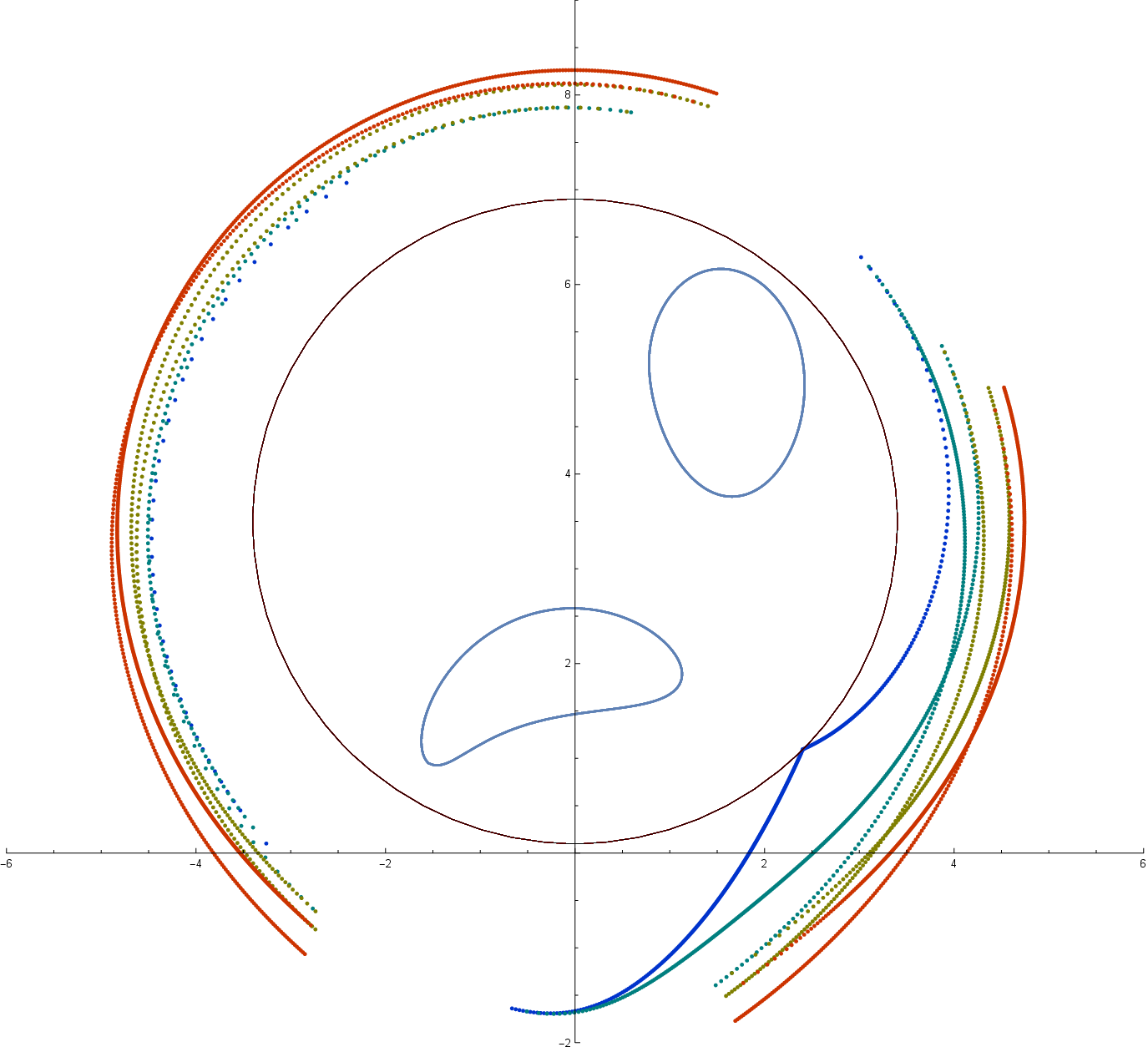}
		\caption{Echograph of two obstacles}
		\label{fig:echograph}
	\end{figure*}
	
	Define the modified arcs $\alpha_i^* = \alpha_i\backslash\cup_{i'\neq i}\alpha_{i'}^*$, to exclude the finite number of points of intersection, and let $\sT^{0}(x_1)^*=\cup_{i\geq 1}\alpha_i^*$. Then the $\alpha_i^*$ partition $\sT^{0}(x_1)^*$ and the generators $\tau_i$ restrict to smooth functions on the open subsets $\pi_1(\alpha_i^*)\subseteq \partial S$. We now define the sets of travelling times with initial directions included.
	\[\widetilde\sT^0(x_1)^*\ = \{(x,-\nabla\tau_i(x),t):(x,t)\in\sT^0(x_1)^*\}\]
	\[\widetilde\sT^{0*} = \{(x,\omega,x_1,t):x_1\in\partial S, (x,\omega,t)\in\widetilde\sT^0(x_1)^*\}\]
	Moreover, define $\widetilde\sT^0$ as the closure of $\widetilde\sT^{0*}$ in \[\{(x,\omega,x_1,t):(x,x_1,t)\in\sT^0\text{ and }(x,\omega)\in\T{S}\},\] and $\wsT^0(x_1) = \{(x,\omega,t):(x,\omega,x_1,t)\in\wsT^0\}$. Similarly, define $\wsT$ as the closure of $\wsT^0$ in \[\{(x,\omega,x_1,t):(x,x_1,t)\in\sT\text{ and }(x,\omega)\in\T{S}\},\] $\wsT(x_1) = \{(x,\omega,t):(x,\omega,x_1,t)\in\wsT\}$ and $\wsT^q = \{(x,\omega,x_1,t)\in\wsT:(x,x_1,t)\in\sT^q\}$.
	
	\begin{proposition}\label{lemma:tangent_arcs}
		$\wsT^1$ is a countable union of maximal smooth open arcs $\beta_j$ such that:
		\begin{itemize}
			\item Each $\beta_j$ is diffeomorphic to a smooth open arc $V_j\subseteq \partial S$
			\item $\wsT^2 = \cup_{j\geq 1}\partial\beta_j$
			\item Each $(x_0,\omega_0,x_1,t)\in\wsT^2$ is on the boundary of exactly four distinct arcs $\beta_j$, $\beta_{j'}$, $\beta_{j''}$, $\beta_{j'''}$ with three of $V_{j}, V_{j'}, V_{j''}, V_{j'''}$ being on the same side of $x_0$.
		\end{itemize}
	\end{proposition}
	\begin{proof}
		Given $(x_0,\omega_0,x_1,t)\in\wsT^1$ we know that the geodesic $\gamma_{(x_0,\omega_0)}$ is tangent to $\partial K$ exactly once, at either the first or last point of contact with $\partial K$ (by general position). Suppose the geodesic is tangent at the first point of contact, say at $x'\in\partial K_i$. Parameterise $\partial K_i$ in a neighbourhood $W$ of $x'$ by the unit speed curve $k_i:(0,1)\to S$. Then by the same argument as \Cref{lemma:singlereflection}, after possibly shrinking $W$ there is a smooth function $\tau:W\to\reals$ such that $y(x) = \flow_{\tau(x)}(x,\dot k_i(k_i^{-1}(x)))\in \T{S}$ for all $x\in W$, and $y(x') = (x_0,-\omega_0)$. This defines a smooth open arc \[V = \{(v(x), \mathcal{P}(v(x)), t(v(x)):x\in W\}\subseteq \wsT^1,\] where $v(x) = (\pi_1\circ y(x),-\pi_2\circ y(x))$ and $t$ is the travelling time function. The symmetric argument would follow if $\gamma_{(x_0,\omega_0)}$ was tangent at the last point of contact. Hence the path components $\beta_j$ of $\wsT^1$ are disjoint maximal smooth open arcs. Note that $\beta_j$ is therefore diffeomorphic to the open arc $V_j = \pi_1(\beta_j)\subseteq \partial S$.
		
		Now for $(x_0,\omega_0,x_1,t)\in\wsT^2$, the geodesic $\gamma_{(x_0,\omega_0)}$ is tangent to $\partial K$ exactly twice, at both the first and last points of contact with $\partial K$, say $x'\in\partial K_i$ and $x''\in\partial K_{i'}$. Since the set $\wsT^2$ is discrete, we may take a neighbourhood $U$ of $x'$ such that every $x\in U$ is contained in a maximal arc $\beta_j$ by the argument above. This defines two maximal arcs $\beta_j$ and $\beta_{j'}$, with $V_j$ and $V_{j'}$ on opposite sides of $x_0$. We can do the same for $x''$ with $\beta_{j''}$ and $\beta_{j'''}$ being maximal arcs and $V_{j''}$ and $V_{j'''}$ on the same side of $x_0$. Note that the four maximal arcs are all distinct. Thus $(x_0,\omega_0,x_1,t)$ is the endpoint of exactly four open arcs.
		
		It now follows that $\wsT^1 = \cup_{j\geq 1}\beta_j$, and $\wsT^2 = \cup_{j\geq 1}\partial\beta_j$.
	\end{proof}
	
	\section{Reconstructing $K$ from $\sT$}
	
	Note that by \Cref{lemma:number_of_obstacles} there are\footnote{Recall that in the case where $n=2$ there are at most 8 arcs.} $4n(n-1)$ arcs $\beta_j$ which are tangent to $\partial K$ but do not intersect the obstacles elsewhere. We re-order the maximal arcs so that for $1\leq j\leq 4n(n-1)$ the arcs $\beta_j$ are the aforementioned arcs arcs precisely.
	For each arc $\beta_j$, take a normal coordinate neighbourhood $U$ containing $V_j$, and denote the diffeomorphism $\psi_j:V_j\to\beta_j$. Let $\xi_{x}:\reals\to S$ be the smooth geodesic starting from $x\in V_j$ in the direction $\psi_j(x)$. Locally the smooth geodesics $\xi_x$ intersect pairwise at exactly one point within $U$. Therefore their envelope, $\Sigma(\beta_j)$, is a smooth curve defined locally in $U$.
	
	Denote by $x^*\in \partial S$ the clockwise terminal limit of $V_j$, and let \[\eta_j = \lim_{x\to x^*}\psi_j(x)\]
	Suppose that $\Sigma(\beta_j)$ is strictly convex, then we say that $\beta_k$ extends $\beta_j$ if $\partial\beta_j\cap\partial\beta_k = \{\eta_j\}$ and the closure of $\Sigma(\beta_j)\cup\Sigma(\beta_k)$ is a strictly convex curve in $S$. Otherwise we say that $\beta_j$ is non-extendible. Note that for each $\beta_j$ there are at most three arcs which could extend it, by \Cref{lemma:tangent_arcs}.
	\begin{proposition}\label{lemma:extendibility}
		Suppose that for every $(x,\omega,x_1,t)\in\beta_j$ the geodesic $\gamma_{(x,\omega)}$ is tangent at the first point of contact with $\partial K$. Then precisely one of the following holds:
		\begin{enumerate}
			\item $\beta_j$ is uniquely extendible, by some arc $\beta_k$ such that and $\Sigma(\beta_j)\cup\Sigma(\beta_k)$ is an arc in $\partial K$.
			\item $\beta_j$ is non-extendible with $1\leq j\leq 4n(n-1)$ and the conjugate arc $\beta_{j^*}$ is extendible, where \[\beta_{j^*} = \{(x_1,\omega,x,t):(x,\omega,x_1,t)\in\beta_j\}\]
		\end{enumerate}
	\end{proposition}
	\begin{proof}
		Given $\beta_j$, and $\eta_j$ corresponding to the clockwise terminal limit $x^*$ of $V_j$, the geodesic $\gamma_{\eta_j}$ is tangent to $\partial K$ exactly twice, by \Cref{lemma:tangent_arcs}. For each $x\in V_j$, let $\widetilde{x}$ be the point of tangency with $\partial K$ of the geodesic generated by $\gamma_{\psi_j(x)}$. Let $\widetilde{x}^*$ be the first point of tangency of $\gamma_{\eta_j}$. Then by continuity of $\psi_j$, the tangency $\widetilde{x}^*$ is a limit of the points of tangency $\widetilde{x}$. There are two possible cases:
		
		\emph{Case 1:} The $\widetilde{x}$ and $\widetilde{x}^*$ remain on the same connected component $\partial K_i$. First note that by \Cref{lemma:tangent_arcs} there are three arcs $\beta_{j'}, \beta_{j''},$ and $\beta_{j'''}$ which could extend $\beta_j$. We suppose that $V_{j'}$ is on the opposite side of $x^*$ from $V_j$, and $V_{j''},V_{j'''}$ are either on the opposite or the same side as $V_{j}$. Now since the first point of contact of $\gamma_{\eta_j}$ remains on $\partial K_i$ it follows that $\beta_{j'}$ will generate geodesics which are also tangent to $\partial K_i$ at the first point of contact. Thus $\beta_{j'}$ extends $\beta_j$, and since both arcs generate geodesics which are tangent to $\partial K_i$, the envelopes $\Sigma(\beta_j)\cup\Sigma(\beta_{j'})$ form an arc in $\partial K_i$. Now to show that $\beta_{j'}$ is the unique arc which extends $\beta_j$, suppose that $\beta_k$ also extends $\beta_j$ where $k = j''$ or $k = j'''$. Note that the arc $\beta_k$ generates geodesics which are tangent to some other connected component $\partial K_{i'}$. Construct convex fronts for the envelope $\Sigma(\beta_k)$ and $\partial K_{i'}$ around the points of tangency from $V_k$, as in \Cref{thm:convexlemma}. Then every geodesic which hits one front orthogonally must hit the other front orthogonally as well, by construction. This gives a contradiction by \Cref{thm:confrontorth}. So $\Sigma(\beta_k)$ cannot be a convex arc in $S$. Thus the extension $\beta_{j'}$ is unique.
		
		\emph{Case 2:} The $\widetilde{x}$ and $\widetilde{x}^*$ do not remain on the same connected component $\partial K_i$. In this case the second point of tangency of $\gamma_{\eta_j}$ remains on the same connected component as the $\widetilde{x}$, but this would imply that the first two points of contact of $\gamma_{\eta_j}$ with $\partial K$ are points of tangency. Hence, by general position, $\gamma_{\eta_j}$ is a smooth geodesic, so $1\leq j\leq 4n(n-1)$. Note that this also implies that the $\widetilde{x}$ were the only points of contact between $\gamma_{\psi_j(x)}$ and $\partial K$. Now to extend $\beta_j$ we look at the conjugate arc $\beta_{j^*}$. By the first case, $\beta_{j^*}$ is extendible, but $\beta_j$ would not be extendible since the arcs which could extend it do not generate geodesics that are tangent to $\partial K_i$.
	\end{proof}
	
	We can now outline how to reconstruct the obstacles from the maximal arcs $\beta_j$. We begin by noting that for any of the first $4n(n-1)$ arcs, we have $\Sigma(\beta_j) = \Sigma(\beta_{j^*})$, where $\beta_{j^*}$ is the conjugate arc defined in \Cref{lemma:extendibility}. Now starting with any of these initial arcs $\beta_j$, we extend $\beta_j$ or $\beta_{j^*}$ according to \Cref{lemma:extendibility}. Each extension can be further extended by the same method. We continue countably many times, or until the envelope of the extending arc becomes acceptably small. Repeating this process for each of the initial $4n(n-1)$ arcs will recover $\partial K$ from the travelling times.

	\begin{acknowledgements}
		This research is supported by an Australian Government Research Training Program (RTP) Scholarship.
	\end{acknowledgements}

	\bibliography{bibliography}
\end{document}